\documentclass[11pt, a4paper]{amsart}

\usepackage{amsmath}
\usepackage{tikz}
\usepackage{amssymb}

\usepackage{amsthm}%%% it must be load
\usepackage{tikz-cd}
\usepackage[all]{xy}
\usepackage{amsfonts}
\usepackage{mathrsfs}%%% to use \mathscr
\usepackage[scr=boondox, cal=esstix]{mathalpha}%%% for lower case. See https://tex.stackexchange.com/questions/668241/lowercase-mathscr-and-mathcal-letters.
\usepackage{hyperref}%%% for hyperlinks
\usepackage{xcolor}
\usepackage{MnSymbol}% provides the ``end of def'' symbol
\usepackage{mathtools}
\usepackage{subcaption}

\usepackage{comment}
\input{header_labeling}

\newcommand{\DD}{\mathbb{D}}

\newcommand{\NN}{\mathbb{N}}

\newcommand{\RR}{\mathbb{R}}

\newcommand{\ZZ}{\mathbb{Z}}

\newcommand{\cA}{\mathcal{A}}

\newcommand{\cD}{\mathcal{D}}

\newcommand{\cF}{\mathcal{F}}
\newcommand{\cG}{\mathcal{G}}
\newcommand{\cH}{\mathcal{H}}

\newcommand{\cK}{\mathcal{K}}
\newcommand{\cL}{\mathcal{L}}
\newcommand{\cM}{\mathcal{M}}
\newcommand{\cN}{\mathcal{N}}
\newcommand{\cO}{\mathcal{O}}

\newcommand{\cT}{\mathcal{T}}

\newcommand{\fB}{\mathsf{B}}
\newcommand{\fC}{\mathsf{C}}

\newcommand{\fF}{\mathsf{F}}

\newcommand{\fH}{\mathsf{H}}

\newcommand{\fK}{\mathsf{K}}

\newcommand{\fM}{\mathsf{M}}
\newcommand{\fN}{\mathsf{N}}

\newcommand{\fR}{\mathsf{R}}
\newcommand{\fS}{\mathsf{S}}

\newcommand{\fU}{\mathsf{U}}
\newcommand{\fV}{\mathsf{V}}
\newcommand{\fW}{\mathsf{W}}
\newcommand{\fX}{\mathsf{X}}

\newcommand{\cc}{\textbf{c}}
\newcommand{\dd}{\textbf{d}}
\newcommand{\kk}{\textbf{k}}
\newcommand{\nn}{\textbf{n}}

\newcommand{\RP}{\operatorname{\RR P}}

\newcommand{\Homeo}{\operatorname{Homeo}}

\newcommand{\Diff}{\operatorname{Diff}}

\newcommand{\Mod}{\operatorname{Mod}}
\newcommand{\PMod}{\operatorname{PMod}}

\newcommand{\closure}[1]{\overline{#1}}
\newcommand{\interior}{\operatorname{Int}}

\newcommand{\inj}{\operatorname{inj}}
\newcommand{\Bl}{\mathscr{Bl}}
\newcommand{\PH}{\operatorname{P\cH}}

\newcommand{\HHH}{\mathrm{H}}

\title[Quasimorphisms on $\Diff_{\omega}(M,\partial M)_0$]{\small{Quasimorphisms on the group of density preserving diffeomorphisms of the M\"{o}bius band}}
\author{KyeongRo Kim}
\address{\hskip-\parindent
Research institute of Mathematics\\
Seoul National University\\
GwanAk-Ro 1, GwanAk-Gu, Seoul 08826, Korea}
\email{kyeongrokim14@gmail.com}

\author{Shuhei Maruyama}
\address{\hskip-\parindent
School of Mathematics and Physics, College of Science and Engineering\\
Kanazawa University\\
Kakuma-machi, Kanazawa, Ishikawa, 920-1192, Japan}
\email{smaruyama@se.kanazawa-u.ac.jp}

\date{\today}

\begin{document}
\maketitle
\begin{abstract}
    The existence of quasimorphisms on groups of homeomorphisms of manifolds has been extensively studied under various regularity conditions, such as smooth, volume-preserving, and symplectic. However, in this context, nothing is known about groups of `area'-preserving diffeomorphisms on non-orientable manifolds.

    In this paper, we initiate the study of groups of density-preserving diffeomorphisms on non-orientable manifolds. 
    Here, the density is a natural concept that generalizes volume without concerning orientability.   
    We show that the group of density-preserving diffeomorphisms on the M\"obius band admits countably many homogeneous quasimorphisms which are linearly independent.
    Along the proof, we show that groups of density preserving diffeomorphisms on compact, connected, non-orientable surfaces with non-empty boundary are weakly contractible.  
\end{abstract}

\section{Introduction}

Let $\Diff(S)_0$ be the identity component of the group of smooth diffeomorphisms of a surface $S$.
Bowden, Hensel and Webb \cite{BowdenHenselWebb22} introduced the fine curve graph of closed orientable surfaces, and proved its Gromov-hyperbolicity.
Furthermore, they used it and the theorem of Bestvina--Fujiwara \cite{BestvinaFujiwara02} to prove that $\Diff(S)_0$ admits infinitely many linearly independent homogeneous quasimorphisms if $S$ is a closed surface of genus greater than $0$.
After that, Kimura and Kuno \cite{KimuraKuno21} proved the similar statement for closed non-orientable surfaces of genus greater than $2$.

About the surfaces with boundary, let $\Homeo(S, \partial S)_0$ be the identity component of the group of homeomorphisms of $S$ which are identity on the boundary $\partial S$.
Bowden, Hensel and Webb \cite{BowdenHenselWebb24} proved that the group $\Homeo(S, \partial S)_0$ admits a homogeneous quasimorphism if the Euler characteristic $\chi(S)$ of $S$ is negative.
B\"{o}ke \cite{Boke24} used this to prove that $\Diff(N_2)_0$ admits a homogeneous quasimorphism, where $N_2$ is the Klein bottle.

As the above quasimorphisms have a geometric group theoretical and hyperbolic nature, the low genus surfaces do not show up.
In fact, it is known that $\Diff(S^2)_0$ is uniformly perfect \cite{BuragoIvanovPolterovich08, Tsuboi08}, and hence does not admit any homogeneous quasimorphisms.

Meanwhile, the situation of the group of area-preserving diffeomorphisms is a bit different.
Let $\omega$ be an area form of a closed orientable surface $S$ and $\Diff_{\omega}(S)_0$ (resp. $\Diff_{\omega}(S, \partial S)_0$) be the identity component of the group of area-preserving diffeomorphisms of $S$ (resp. which are identity on the boundary). 
By using the Floer and quantum homology, Entov and Polterovich \cite{EntovPolterovich03} constructed uncountably many homogeneous quasimorphisms on $\Diff_{\omega}(S^2)_0$, which are linearly independent.
Gambaudo and Ghys \cite{GambaudoGhys04} used a dynamical construction of braids and its signature to construct countably many homogeneous quasimorphisms on $\Diff_{\omega}(D,\partial D)_0$, which are linearly independent.
Here $D$ denotes the closed unit disk.
The construction of Gambaudo and Ghys has been studied and generalized for orientable surfaces with higher genus (\cite{Brandenbursky11}, \cite{Ishida14}, \cite{Kimura20}).

However, there is nothing known about `area'-preserving diffeomorphism groups of non-orientable surfaces.
This article is a first step in studying quasimorphisms on the groups of `area'-preserving diffeomorphisms of non-orientable surfaces.
We consider the construction of Gambaudo and Ghys on the M\"{o}bius band $M$, which is the non-orientable surface with boundary of lowest genus.

Since the usual area form is not well-defined on $M$, we instead use the standard density $\omega$ of $M$, which is a nowhere vanishing differential $2$-form of odd type, or twisted differential $2$-form (see \refsec{oddForm} for the definition).
Let $\Diff_{\omega}(M,\partial M)_0$ be the identity component of the group of density-preserving diffeomorphisms.
We prove the following.
\begin{restate}{Theorem}{Thm:inftyDim}
    The group $\Diff_{\omega}(M,\partial M)_0$ admits countably many  homogeneous quasimorphisms which are linearly independent.
\end{restate}

To follow the dynamical construction of braids of Gambaudo and Ghys, we need the weak contractibility of $\Diff_\omega(M,\partial M)_0$.
For future reference, we prove the weak contractibility on general compact surfaces.
The case of orientable surfaces has been shown by Tsuboi \cite{Tsuboi00}.
\begin{restate}{Theorem}{Thm:contractible}
Let $F$ be a compact, connected surface with non-empty boundary, possibly non-orientable. Then, $\Diff_\omega(F,\partial F)_0$ is weakly contractible.
\end{restate}
\begin{rmk}
    Here, $\omega$ is a density form on $F$. By a version of Moser's theorem, \cite{Bruveris}, the choice of a certain density form is not significant. In particular, if the underlying manifold is orientable, then any density form naturally corresponds to an area form. 
\end{rmk}

\subsection*{Origanization and Strategy}
In this paper, we deal with non-orientable manifolds. 
In non-orientable manifolds, there is no volume form in usual sense. 
Hence, we use a ``twisted" version of forms which is introduced by de Rham \cite{deRham84}, and  density forms instead of volume forms. 
In \refsec{basicNotion}, we review the theory of twisted de Rham cohomology. 
Also, we recall the basic notations and set conventions used throughout the paper.

In \refsec{exact}, we show the exactness of the density forms on non-orientable manifolds with non-empty boundary. 
This fact is well-known for orientable manifolds.  

In \refsec{fiberContractible}, we show that in \refthm{contractible}, the simply connectedness of $\Diff_{\omega}(F,\partial F)_0$. 
This is necessary to ensure the well-defineness of the Gambaudo-Ghys type cocycles. 

In \refsec{braidGroupOnM}, we recall basic notions about the mapping class groups and braid groups on the M\"obius bands.
Also, we discuss the algebraic and topological propeties of those groups.
In particular, we show that every pure braid group and braid group with at least two strands admit countably many homogeneous quasimorphisms which are linearly independent (\reflem{inftyDimBraid}).

The main goal of  \refsec{dimension} is to prove the main theorem, \refthm{inftyDim}.
To do this, we define a homomorphism $\cG:Q(B_2(M))\to Q(\Diff_\omega(M,\partial M)_0)$, following the construction of Gambaudo-Ghys where $Q(G)$ denotes the space of homogeneous quasimorphisms over $G$.
The main theorem is shown by the injectivity of $\cG$ (\refthm{injectivity}).

To show the injectivity of $\cG$, we follow Ishida's strategy \cite{Ishida14}(or more generally \cite{Brandenbursky15}). 
Hence, we construct a sequence of density-preserving representatives of a given Dehn twist while keeping conjugation-generated norms of the associated Gambaudo-Ghys type cocycle within some bounded error (\reflem{canonicalDehnTwist}).
Unlike in \cite{Ishida14} and \cite{Brandenbursky15}, estimating the conjugation-generated norms is not straightforward.
Hence, we adapt the concept of a fibered surface, introduced in \cite{BestvinaHandel95} and carfully construct representatives of a given Dehn twist.

Along the proof of the well-definedness of $\cG$, we make use of the word-length estimation of Gambaudo-Ghys cocycles (\reflem{finiteness}).
To do this, in \refsec{wordLen}, we introduce a blowing-up technique to compactify the configuration space $X_2(M)$ of the distinct two points in the M\"obius band $M$. 
Then, we provide the proof of the word-length estimation of Gambaudo-Ghys cocycles (\reflem{finiteness}), using the blowing-up technique.
Our blowing-up set is a refinement of the blowing-up set, introduced by Gambaudo and P\'ecou \cite{GambaudoPecou99}.
Unlike in \cite{GambaudoPecou99}, our blowing-up set does not change the topology of $X_2(M)$ (\reflem{homEqu}).
This compactification is essentially the same as the compactification introduced in \cite[Section~2]{Brandenbursky22}.

\section{Preliminary}\label{Sec:basicNotion}
\subsection{Conjugation-generated norms}
Let $G$ be a group and $S$ a finite subset of $G$.
We say that $S$ \emph{finitely generates} $G$ if every element $g\in G$ can be written as a product 
\[
g=g_1g_2 \cdots g_N
\]
where one of $g_i$ and $g_i^{-1}$ is an element in $S$.
The minimal possible $N$ for such products is called   the \emph{word length} of $g$ with respect to $S$ and it is denoted by $\ell_S(g)$.

In \cite[Section~1.2.1]{BuragoIvanovPolterovich08}, they introduced conjugation-generated norms as follows. 
We say that $S$ \emph{finitely conjugation-generates}(or finitely c-generates) $G$ if every element $g\in G$ can be written as a product 
\[
g=g_1g_2 \cdots g_N
\]
where one of $g_i$ and $g_i^{-1}$ is conjugate to an element in $S$.
Also, the minimal possible $N$ for such products is denoted by $q_S(g)$.
We say that the norm $q_S$ is \emph{c-generated} by $S$. 

\subsection{Quasimorphisms}
In this subsection, we recall the definition and some properties of quasimorphism.
We refer the reader to \cite{Calegari09} for details.
A real valued function $\mu \colon G \to \RR$ on a group $G$ is called a \emph{quasimorphism} if there exists a non-negative constant $D$ such that for every $g,h \in G$, the inequality \[|\mu(gh) - \mu(g) - \mu(h)| \leq D\] holds.
Also, the minimum value of such a $D$  is called  the \emph{defect} of $\varphi$, denoted by $D(\varphi)$.
A quasimorphism $\mu$ is said to be \emph{homogeneous} if $\mu(g^k) = k \mu(g)$ holds for every $g \in G$ and $k \in \ZZ$.
Let $Q(G)$ denote the space of homogeneous quasimorphisms over $G$.

The homogeneity condition is not so restrictive. In fact, for every quasimorphism $\mu$, the function $\overline{\mu} \colon G \to \RR$ defined by 
\[
    \overline{\mu} (g) = \lim_{p \to +\infty} \frac{\mu(g^p)}{p}
\]
is a homogeneous quasimorphism and the difference $\overline{\mu} - \mu$ is a bounded function.
In particular, the existence of unbounded quasimorphisms is equivalent to the existence of homogeneous quasimorphisms.
We call $\overline{\mu}$ the \emph{homogenization of} $\mu$.

In the last part of the proof of \refthm{inftyDim}, we use the following basic fact.
\begin{prop}[{See \cite[Subsection 2.2.3]{Calegari09}}]\label{Prop:conjInv}
    Every homogeneous quasimorphism $\mu \colon G \to \RR$ is invariant under conjugation.
\end{prop}
As an immediate corollary,  we have the following proposition:
\begin{prop}\label{Prop:qValue}
    If $G$ is finitely $c$-generated by $S$ and  $\mu:G\to \RR$ is a homogeneous quasimorphism, then  we have
    \[
    \mu(g) \leq (\fM+D(\mu)) q_S(g)
    \]
    for all $g\in G$,
    where $\fM=\max \{\mu(s^{\pm1}):s\in S\}$. 
\end{prop}

\subsection{Twisted differential forms}\label{Sec:oddForm}
In \cite{deRham84}, de Rham  introduced the differential form of odd type. 
In \cite{BottTu82}, Bott and Tu also discussed this in terms of twisted de Rham complex.
We refer to the books \cite{deRham84} and \cite{BottTu82} for the detailed expositions about elementary algebraic topological facts with differential forms of odd type, e.g. Stokes' theorem.

Recall the definition of the orientation bundle  of a smooth manifold $\cN$ with or without boundary.
We denote it by $L_{\cN}$. 
Also, recall the \emph{trivialization induced from the atlas $\{(U_\alpha, \phi_\alpha)\}$ on $\cN$} and the \emph{standard locally constant sections} (see \cite[page~84]{BottTu82} and \cite[page~80]{BottTu82}, respectively).
From now on, whenever we mention a trivialization of the orientation bundle, it refers to the trivialization induced from a given atlas.

For the simplicity, we call an $L_\cN$-valued differential $p$-form a \emph{twisted differential $p$-form}, and we let $\Omega^p(\cN;L_\cN)$ denote the set of twisted differential $p$-forms over $\cN$.
Note that the twisted differential forms are equivalent objects to the differential forms of odd type in \cite{deRham84}.
A \emph{density form} of $\cN$ is a twisted $(\dim \cN)$-form which is nowhere zero.

One of the most tricky parts in \cite{deRham84} is to define a pullback of a differential forms of odd type by some smooth map $h$. 
To do this, we need a converting rule between standard locally constant sections of the domain and range of $h$.
Thus, we include some exposition about a pullback.

Let $\cN$ and $\cM$ be connected smooth manifolds with or without boundary of dimension $n$ and $m$, respectively, possibly non-orientable.
Let $h:\cN \to \cM$ be a smooth map and $\nu$ an $L_\cM$-valued $p$-form in $ \Omega^p(\cM;L_\cM)$. 
To define the pullback of $\nu$ by $h$, we need  a well-defined bundle morphism $h_L:L_\cN \to L_\cM$ such that for any trivializations $(U,\phi)$ and $(V,\psi)$ of $L_\cN$ and $L_\cM$, respectively, with $h(U)\subset V$, if $e_V$ is the standard locally constant section of $L_\cM$ over $V$, then the local section $e$ of $L_\cN$ over $U$, defined as 
\[
    h_L(e(x))=e_V(h(x))
\] 
is either the standard locally constant section $e_U$ of $L_{\cN}$ over $U$ or $-e_U$.
If there is such an $h_L$, then $h$ is said to be \emph{orientable} and if such an $h_L$ is fixed, then $h$ is said to be \emph{oriented} by $h_L$.
In this case, the \emph{pullback} $    h^*\nu$ of $\nu$ by $h$ with respect to $h_L$ is defined as 
\[
    (h^*\nu)_x= h^*v \otimes h_L^{-1}(e)
\]
for $v\in (\bigwedge^p T^*\cM)_{h(x)}$ and $e\in L_{h(x)}$ with $\nu=v\otimes e$.
Note that $h_L$ is the concept corresponding to the \emph{orientation} of a map $h$ in \cite{deRham84}.

In particular, if $n=m$ and $h$ has no critical point, then there is a canonical bundle morphism $h_L:L_\cN \to L_\cM$ such that for any trivializations $(U,\phi)$ and $(V,\psi)$ of $L_\cN$ and $L_\cM$, respectively, such that $h(U)\subset V$ and the Jacobian determinant of $\psi\circ h\circ \phi^{-1}$  is positive on $\phi(U)$, if $e_V$ is the standard locally constant section of $L_\cM$ over $V$, then the local section $e$ of $L_\cN$ over $U$, defined as 
\[
    h_L(e(x))=e_V(h(x))
\] is equal to the standard locally constant section $e_U$ of $L_{\cN}$ over $U$. 
In this case, the map $h_L$ is the same thing with the \emph{canoical orientation} of the map $\iota$, introduced in \cite[page~21]{deRham84}.
From now on, we use the canonical orientation without mentioning if there is no confusion.

\subsection{M\"{o}bius band}
In this subsection, we fix some notations about the M\"{o}bius band.
From now on, whenever we mention  $M$, we refer to the closed M\"obius band.
Also, we use the following conventions without further mention  in the rest of the paper.

We set $I=[-1/2,1/2]$ and $\widetilde{M}:= \RR \times I$.
Let $\tau:\widetilde{M} \to \widetilde{M}$ be the deck transformation defined as 
\begin{align*}
    \tau(\begin{bmatrix} z \\ w \end{bmatrix})=\begin{bmatrix} 1 && 0 \\ 0 && -1 \end{bmatrix}\begin{bmatrix} z \\ w \end{bmatrix} + \begin{bmatrix} 1 \\ 0 \end{bmatrix}.
\end{align*}
The M\"obius band $M$ is defined by $\widetilde{M}/\langle \tau \rangle$. 
Let $\pi:\widetilde{M}\to M$ be the quotient map.
For convenience, we also define the \emph{M\"obius band $M_r$ with width $r$} as $M_r=\pi(\RR\times [-r/2,r/2])$.
Note that $M_1=M$.

For small $\epsilon$, we set 
\begin{align*}
    \fU :=& (-1/2-\epsilon, \epsilon)\times I,\\
    \fV :=& (-\epsilon, 1/2+\epsilon)\times I,\\
    \fW_0 :=& (-1/2-\epsilon, -1/2+\epsilon)\times I,\\
    \fW_1 :=& (-\epsilon, \epsilon)\times I,\\ 
    \fW_2 :=& (1/2-\epsilon, 1/2+\epsilon)\times I.
\end{align*}
Let $U:=\pi(\fU)$ and $V:=\pi(\fV)$, which cover $M$. 
Also, write $W_0=\pi(\fW_0)=\pi(\fW_2)$ and $W_1=\pi(\fW_1)$.
    
For coordinate maps, set
\begin{align*}
    \varphi_U:&U\to \fU,\varphi_U:=(\pi|_\fU)^{-1}, \\
    \varphi_V:&V\to \fV,\varphi_V:=(\pi|_\fV)^{-1}.
\end{align*}
The connection for the line bundle $L_M$, $g_{UV}: U\cap V \to \{\pm 1\}$, is defined as $g_{UV}(w)=-1$ if $w\in W_0$ and $g_{UV}(w)=1$ if $w\in W_1$.
The local sections are given by
\begin{align*}
    e_U:&U\to U\times \RR, e_U:w \mapsto (w,1), \\
    e_V:&V\to V\times \RR, e_V:w \mapsto (w,1). 
\end{align*}
Then a density form $\omega \in \Omega^2(M,L_M)$ is defined by
\begin{align*}
    (\varphi_U^{-1})^*\omega&:=(dx \wedge dy) \otimes e_U\\
    (\varphi_V^{-1})^*\omega&:=(du \wedge dv) \otimes e_V
\end{align*}
where $(x,y)\in \fU$ and $(u,v) \in \fV$.

\section{Exactness of density forms}\label{Sec:exact}

De Rham showed the homotopy invariance for homology groups of currents (which are generalizations of singular chains and differential forms). 
See \cite[$\mathsection 18.$ Homology Groups]{deRham84}.
We rephrase the theorem for our purpose as follows:
\begin{prop}[Homotopy invariance of twisted de Rham cohomologies]
    Let $\cN,\cM$ be compact, connected, smooth manifolds, possibly non-orientable, and $F,G$ smooth maps from $\cN$ to $\cM$. If there is a smooth homotopy $H:\cN\times [0,1]\to \cM$ from $F$ to $G$ and $H$ is oriented, then 
    for all $i\geq 0$, the induced homomorphisms
    $F^*, G^*:\HHH^i(\cM;\cL_\cM)\to \HHH^i(\cN;\cL_\cN)$ coincide.
\end{prop}

Let $\cN$ be a compact, connected $n$-manifold with non-empty boundary, possibly non-orientable. 
We denote the interior of $\cN$ by $\interior (\cN)$.
\begin{prop}\label{Prop:equiToInt} $\HHH^i(\cN;L_\cN)\cong \HHH^i(\interior (\cN);L_{\interior (\cN)})$ for all $i\in \ZZ_{\geq 0}$.
\end{prop}
\begin{proof}
 From  \cite[Theorem~9.26]{Lee13} and its proof, we can see that there is a proper smooth embedding $R:\cN \to \interior (\cN)$ such that 
 both $\iota\circ R:\cN \to \cN$ and $R\circ \iota: \interior (\cN) \to \interior (\cN)$ are smoothly homotopic to the identities,
 where  $\iota : \interior (\cN) \to \cN$ is the inclusion map.
Moreover, the homotopies can be oriented in a canonical way.
Therefore, by the homotopy invariance of twisted de Rham cohomologies, we can obtained  the desired results. 
\end{proof}

Then, we observe that every density form $\omega$ in $\cN$ is exact.
\begin{lem}\label{Lem:exact} 
    There is a twisted $(n-1)$-form $\eta$ such that $d\eta=\omega$. 
\end{lem}
\begin{proof}
    When $\cN$ is orientable, it is already known. 
    Assume that $\cN$ is non-orientable.
    To see this, it is enough to show that $\HHH^n(\cN;L_\cN)=0$.
    It follows from the following equalities: 
    \[\HHH^n(\cN;L_\cN)\cong \HHH^n(\interior (\cN);L_{\interior (\cN)})\cong \HHH^0_c(\interior (\cN))=0.
    \]
    The first equality comes from \refprop{equiToInt}, and the second equality follows from the Poincar\'e duality (e.g. \cite[Theorem~7.8]{BottTu82}).
    Then, the third one is obtained by the direct computation 
    since $\interior (\cN)$ is a connected, non-compact manifold.
\end{proof}

\section{Contractibility of the identity component}\label{Sec:fiberContractible} 
Now, we prove \refthm{contractible}, which allows us to define the Gambaudo-Ghys type cocycles (\refsec{GGcocycle}).
Let $\cN$ be a connected manifold with non-empty boundary. When $\cN$ is orientable, Tsuboi showed that the homotopy fiber of $\Diff_\Omega(\cN, \partial \cN)_0\to \Diff(\cN,\partial \cN)_0$
is weakly contractible for an orientable manifold $\cN$. 
We  follow the argument in \cite[Proposition~2.4]{Tsuboi00}:
\begin{prop}\label{Prop:fiberContractible}
Let $\cN$ be a connected, compact manifold with non-empty boundary $\partial \cN$, that is possibly non-orientable. The homotopy fiber  of $$\Diff_\omega(\cN, \partial \cN)_0\to \Diff(\cN,\partial \cN)_0$$ is weakly contractible.
\end{prop}

\begin{proof}
    The case where $M$ is orientable is shown by Tsuboi \cite[Proposition~2.4]{Tsuboi00}.
    Assume that $\cN$ is non-orientable.
    In this case, we can think of the orientation bundle of $\partial \cN$ as the restriction of $L_\cN$ to $\partial \cN$. 
    Under this identification, the inclusion map $\iota:\partial \cN \to \cN$ is oriented.
    
    We denote the $n$-disk by $D^{n}$ and its boundary sphere by $S^{n-1}$.
    Choose $p>1$. 
    Let $h:S^{p-1}\to \Diff_\omega(\cN,\partial \cN)_0$ be a smooth map.
    We assume that we have a smooth extension $H:D^p \to \Diff(\cN, \partial \cN)_0$ of $h$, that is, $H{\restriction_{S^{p-1}}}=h$.
    Set \[\omega_t^{(v)}=(1-t)H(v)^*\omega+t\omega\] for all $t\in [0,1]$ and $v\in D^p$.
    Then, by \reflem{exact}, there is a twisted $(\dim(M)-1)$-form $\eta$ such that $d \eta= \omega$. Note that by the Stokes' theorem (e.g. see \cite{deRham84} for twisted differential forms),
    \[
    \int_\cN H(v)^*\omega =\int_{\partial \cN} (H(v){\restriction_{\partial \cN}})^* \eta = \int_{\partial \cN}\eta =\int_\cN \omega 
    \]
    for all $v\in D^p$.
    Put \[\alpha_v=H(v)^*\eta-\eta \text{ and so } d\alpha_v=H(v)^*\omega-\omega  \]
    for all $v\in D^p$.

    By the Collar Neighborhood Theorem (see e.g. \cite[Theorem~9.25]{Lee13}), $\partial \cN$ has a collar neighborhood, namely, there is a smooth embedding $j: \partial \cN 
    \times [0,1] \to \cN$ which restricts to the canonical inclusion map from $ \partial \cN \times 0 \to \partial \cN$.
    The image of $j$ is the collar neighborhood $U$ of $\partial \cN$.
    For the simplicity, we identify $U$ with $\partial \cN \times [0,1]$.
    
    Now, we take a smooth function $\mu$ on $\cN$ that is supported on $U$, is $1$ in a neighborhood of  $\partial \cN \times 0$ and is $0$ on a neighborhood of  $ \partial \cN \times 1$.
    Observe that since $\alpha_v{\restriction_{\partial \cN}}=0$, we can write 
    \[
    \alpha_v=a_v(y,t)\wedge dt + b_v(y,t)\omega_{\partial \cN}
    \]
    for $(y,t)\in \partial \cN\times [0,1]$ 
    where  $\omega_{\partial \cN}$ is the density form of $\partial \cN$ and $b_v(y,0)=0$.
    Put \[\beta_v=\alpha_v-d(\mu \cdot a_b(u,0)t).\]
    Note that $d\beta_v=d\alpha_v=H(v)^*\omega -\omega$ and $\beta_v(z)=0$ for all $z\in \partial \cN$.

    Now, we take the time-dependent vector field $X_t^{(v)}$ such that $i(X_t^{(v)})\omega_t^{(v)}=\beta_v$.
    Let $\varphi_t^{(v)}$ be the time-dependent flow of $\cN$ such that 
    \[
    \frac{\partial \varphi_t^{(v)}}{\partial t}(\varphi_t^{(v)}(z))=X_t^{(v)}(\varphi_t^{(v)}(z)).
    \]
    Then, 
    \begin{align*}
        \frac{\partial}{\partial t}(\varphi_t^{(v)})^* \omega_t^{(v)}&=(\varphi_t^{(v)})^* (L_{X_t^{(v)}}\omega_t^{(v)}+\frac{\partial \omega_t^{(v)}}{\partial t})\\
        &=(\varphi_t^{(v)})^*\left (d\left (i(X_t^{(v)})\omega_t^{(v)}\right )-H(v)^*\omega+\omega\right)\\
        &=0.
    \end{align*}
Therefore, we have that 
$\varphi_0^{(v)}=id_{\cN}$ and 
$(\varphi_1^{(v)})^*\omega=H(v)^*\omega$ for $v\in D^p$, and $(\varphi_t^{(v)})^*\omega=\omega$ for $v\in S^{p-1}$.
Set 
\[
\fH_t(v)=
\begin{dcases}
H(v/\Vert v \Vert) \circ (\varphi_{2t(1-\Vert v \Vert)}^{(v/\Vert v \Vert)})^{-1}  & \text{ for } \Vert v \Vert >1/2,\\
H(2v)(\varphi_t^{(2v)})^{-1} & \text{ for } \Vert v \Vert \leq 1/2. 
\end{dcases}
\]
Then, $\fH_0(v)=H(v)$ for all $v\in S^{p-1}$, $\fH_0(D^p)=H(D^p)$ and  $\fH_1(D^p)\subset  \Diff_\omega(M,\partial M)_0$.
Thus, we can conclude that $\Diff_\omega(M,\partial M)_0$ is weakly contractible.
\end{proof}

Recall that Earle-Schatz \cite{EarleSchatz70} showed the following result.
\begin{thm}
    Let $F$ be a smooth compact surface with boundary, possibly non-orientable. Then, 
    $\Diff(F,\partial F)_0$ is contractible.  
\end{thm}

This theorem,
together with \refprop{fiberContractible},  implies the following contractibility. 
\begin{thm}\label{Thm:contractible}
 Let $F$ be a compact, connected surface with non-empty boundary, possibly non-orientable. Then, $\Diff_\omega(F,\partial F)_0$ is weakly contractible.
\end{thm}

\section{Mapping class groups and Braid groups on the M\"obius band}\label{Sec:braidGroupOnM}
Before proceeding with the proof of \refthm{inftyDim}, we introduce some necessary notions and recall some facts about mapping class groups and braid groups on surfaces.

Let $S$ be a topological space.
For the clarity, we write $S^{\times n}$ for the product of $n$ copies of $S$. 
Also, we write $x_i$ for the $i$-th entry of $x\in S^{\times n}$. 
For a homeomorphism $h$ on $S$,  a homeomorphism $\bar{h}$ on $ S^{\times n}$ is defined as $\bar{h}(z)_i=h(z_i)$.
For any $n>1$, the $n$-th generalized diagonal $\Delta_n(S)$ of $S$ is defined as $$\Delta_n(S)=\{x\in S^{\times n}:x_i= x_j\text{ for some }i\neq j\}.$$
We define $X_n(S)$ as $X_n(S)=S^{\times n}\setminus \Delta_n(S)$ for all $n> 1 $, and set $X_1(S)=S$. 
If  $S$ is a surface equipped with a density form, then the measure induced from the density form  induces a canonical measure on $X_n(S)$.

The \emph{pure braid group} of a manifold $\cN$ with $n$-strands is defined by the fundamental group of  $X_n(\cN)$.
Likewise, the \emph{braid group} of a manifold $\cN$ with $n$-strands is defined by the fundamental group of  $X_n(\cN)/\fS_n$
where $\fS_n$ is the symmetric group of degree $n$, acting on $X_n(\cN)$ as coordinate permutations.

The connected orientable surface of genus $g$ with $b$ boundary components is denoted by $S_g^b$. Likewise, $N_g^b$ represents the connected non-orientable surface of genus $g$ with $b$ boundary components, e.g. $N_1^1$ is the closed M\"obius band. 

Let $F$ be a compact, connected surface and  $P=\{x_1,x_2,\ldots, x_p\}$  be a finite (possibly empty) subset of the interior of $F$. 
If $F$ is orientable, that is, $F=S_g^b$, then $\cH(F,P)$ is the set of orientation-preserving homeomorphisms $h$ of $F$ such that $h(P)=P$ and $h$ is the identity on each boundary component of $F$. 
If $F$ is non-orientable, that is, $F=N_g^b$, $\cH(F,P)$ is the set of homeomorphisms $h$ of $F$ such that $h(P)=P$ and $h$ is the identity on each boundary component of $F$. 
For the convenience, we simply write $\cH(F)$ instead of  $\cH(F,\emptyset)$.

We denote the subgroup of $\cH(F,P)$ preserving $P$ pointwise by $\PH(F,P)$.
Then, $\Mod(F,P)$ is $\pi_0(\cH(F,P))$ and $\PMod(F,P)$ is $\pi_0(\PH(F,P))$.
If the choice of $P$ is not significant, then we denote the set $P$ by its cardinality $p$, abusing the notation, that is, $\Mod(F,P)$ and $\PMod(F,P)$ are denoted by $\Mod(F,p)$ and $\PMod(F,p)$.

\subsection{Braid groups and Mapping class groups of the M\"obius band}\label{Sec:dimensionForBraid}
In this section, we observe that pure braid groups and braid groups on the M\"obius band admits countably many homogeneous quasimorphisms which are linearly independent.

By a small variation of \cite[Theorem~4.3]{McCarty63}, we can obtain the following lemma. 
See also the book of Farb and Margarlit, \cite[Section 9.1.4]{FarbMargalit12}.
\begin{lem}\label{Lem:fibration}
Let $P=\{p_1,p_2,\cdots, p_n\}$ be a finite subset of $\interior{M}$. Then, 
    \[
    \PH(M,P) \xrightarrow{F} \cH(M) \xrightarrow{ev_p} X_n(\interior{M})
    \]
is a fibration where $F$ is the forgetful map and $ev_p(f)=(f(p_1),\cdots,f(p_n))$. Also, 
    \[
    \cH(M, P) \xrightarrow{F}  \cH(M) \xrightarrow{ev_p} X_n(\interior{M})/\fS_n
    \]
is a fibration
where $\fS_n$ is the symmetric group of degree $n$.
\end{lem}

The following lemma was shown by Scott. See \cite[Lemma~0.11]{Scott70b}.
\begin{lem}[Scott]\label{Lem:contractible}
    $\cH(M)$ is contractible. 
\end{lem}

Then, the following corollary follows from the long exact sequences of the fibrations in \reflem{fibration}, together with \reflem{contractible}.

\begin{cor}\label{Cor:equivalence}
    $P_n(M)=\PMod(M,n)$ and  $B_n(M)=\Mod(M,n)$ for all $n\in \NN$.
\end{cor}

In \cite{GoncalvesGuaschi17},  $\Gamma_{m,n}(\RP^2)$ is defined as $P_m(\RP^2\setminus\{x_1,\ldots, x_n\})$.
Observe that $\Gamma_{2,1}(\RP^2)=P_2(M)$. In particular, as in the proof of \cite[Proposition~11]{GoncalvesGuaschi17},  we also know that for $m, n\geq1$,  the following Fadell–Neuwirth short exact sequence of
pure braid groups of $\RP^2 \setminus\{x_1,\ldots,x_n\}$ holds:
\begin{align*}
  1\to 
  P_1(\RP^2\setminus \{x_1,\ldots,x_{n+m}\}) 
  \to \Gamma_{m+1,n}(\RP^2) 
  \xrightarrow{q} 
  \Gamma_{m,n}(\RP^2)
  \to 1,
\end{align*}
where the homomorphism $q$ is given geometrically by forgetting the last string.
\begin{prop}\label{Prop:basicCase}
    $P_2(M)=\Gamma_{2,1}(\RP^2)\cong F_2 \rtimes \ZZ$.
\end{prop}
\begin{proof}
    Consider the Fadell–Neuwirth short exact sequence with $m=n=1$:
    \begin{align*}
  1\to 
  P_1(\RP^2\setminus \{x_1, x_{2}\}) 
  \to \Gamma_{2,1}(\RP^2) 
  \xrightarrow{q} 
  \Gamma_{1,1}(\RP^2)
  \to 1.  
\end{align*}
Thus, the result follows from the facts that $$ P_1(\RP^2\setminus \{x_1, x_{2}\})= \pi_1(\RP^2\setminus \{x_1,x_{2}\})\cong F_2 $$ and $$\Gamma_{1,1}(\RP^2)=P_1(\RP^2\setminus \{x_1\})\cong \ZZ.$$
\end{proof}

\begin{lem}\label{Lem:inftyDimBraid}
For $n\geq 2$, $Q(P_n(M))$ and  $Q(B_n(M))$ are of infinite dimension. 
\end{lem}
\begin{proof}
First, we observe that for $n\geq 2$, $P_n(M)=\Gamma_{n,1}(\RP^2)$ is not virtually abelian.
The case of $n=2$ is done by  \refprop{basicCase}. Then, the claim is obtained by an induction argument with the Fadell–Neuwirth short exact sequence with $n=1$.
Since $P_n(M)$ is a finite index subgroup of $B_n(M)$, for $n\geq 2$, $B_n(M)$ are also not virtually abelian.

Once we show that  $P_n(M)=\PMod(M,n)$ and  $B_n(M)=\Mod(M,n)$ are embedded in  $\Mod(S,2n)$ for some closed surface $S$, the result follows from Bestvina-Fujiwara \cite[Theorem 12]{BestvinaFujiwara02} and the fact that $P_n(M)$ and $B_n(M)$ are not virtually abelian.
Therefore, it is enough to show the existence of such a surface $S$.

First, we observe that $\PMod(M,n)$ and $\Mod(M,n)$ are well embedded in $\Mod(A,2n)$ by Katayama-Kuno \cite[Lemma 2.7]{KatayamaKuno24},
where $A$ is the orientation double cover which is an annulus. Then, we attach two one-holed tori on the boundary of $A$ to obtain a genus two surface $S$.
By Paris--Rolfsen \cite[Corollary~4.2]{ParisRolfsen00}, we can see that $\Mod(A,2n)$ is also embedded in $\Mod(S,2n)$. Thus, $S$ is a desired surface.
\end{proof}

\subsection{Twist subgroup}
In the last part of \refthm{injectivity}, we essentially use the concept of twist subgroups discussed in \cite{KatayamaKuno24}.
Let $N=N_g^b$ and $P$ a finite subset of $\interior{N}$.
A simple closed curve in $N\setminus P$ is \emph{peripheral} if it is isotoped to a boundary component in $N\setminus P$.
A two-sided simple closed curve in $N\setminus P$ is \emph{generic} if it does not bound neither a disk nor a M\"obius band in $N\setminus P$ and is not peripheral. 
The \emph{twist subgroup} $\cT(N,P)$ is the subgroup of $\Mod(N,P)$,
 generated by Dehn twists along two-sided closed curves which are either peripheral or generic on $N\setminus P$. See, e.g. \cite[2.~Preliminaries]{Stukow}, for the definition of Dehn twists on non-orientable surfaces.

 \begin{prop}[\cite{KatayamaKuno24}]\label{Prop:twistGrp}
 $\cT(N,P)$ is a finite index subgroup of $\Mod(N,P)$.
 \end{prop}

\section{The dimension of $Q(\Diff_\omega(M,\partial M)_0)$}\label{Sec:dimension}
In this section, we show one of our main theorem, \refthm{inftyDim}. 
The strategy is as follows: 
we first construct some homomorphism $\cG:Q(B_2(M)) \to Q(\Diff_\omega(M,\partial M)_0)$ following Gambaudo-Ghys \cite{GambaudoGhys04} and show that it is well-defined (\refthm{wellDefineness}). 
Then, we show the injectivity of $\cG$ (\refthm{injectivity}).
Finally, \refthm{inftyDim} follows from \reflem{inftyDimBraid} and \refthm{injectivity}.

\subsection{Gambaudo-Ghys type cocycles}\label{Sec:GGcocycle}
Given $g\in \Diff_\omega(M,\partial M)_0$ and given $z\in X_n(M)$, we define the correspoding pure braid $\gamma(g;z)$, following a similar strategy in  \cite[Section~1.1]{Brandenbursky15}. Since $M$ is not contractible, we need to be careful unlike in the case of $D$, to achieve the cocycle condition 
$$\gamma(gh;z)=\gamma(h;z)\cdot \gamma(g;\bar{h}(z))$$
 where $\bar{h}$ is the diagonal action of $h$ in $X_n(M)$. To do this, we choose a ``branch cut" in $M$ as in \cite[Section~2.B.]{BrandenburskyMarcinkowski19}. Let $\ell$ be the line $\pi(1/2\times I)$ and set $\hat{M}=M\setminus \ell$. Then $\hat{M}$
is an embedded disk (with two subarcs of the boundary removed) in $M$ with full measure. Then, any pair of points, $x,y$ in $\hat{M}$, is joined by a unique geodesic path $s_{xy}:[0,1]\to \hat{M}$ from $x$ to $y$ under the canonical Euclidean metric induced from $\widetilde{M}$. 

Fix $n\in \NN$ and a base point $\bar{z}\in X_n(\hat{M})$.
Then, we denote by $\Omega^{2n}$ the set of all points $z$ in $X_n(\hat{M})$ such that $(s_{\bar{z}_iz_i}(t))_{i=1,2,\cdots,n}\in X_n(M)$ for all $t\in [0,1]$.
Since $X_n(\hat{M})$ is an open, dense subset of $X_n(M)$, by a similar argument in \cite[Section~3.2.]{GambaudoPecou99}, we can see that $\Omega^{2n}$ is an open, dense subset of $X_n(M)$ and also that $\Omega^{2n}$ has full measure in $X_n(M)$. 

We are now ready to define the cocycle mentioned above.
For each $g\in \Diff_\omega(M,\partial M)_0$, we define a pure braid $\gamma(g;z)$ in $P_n(M)$, for $z\in \Omega^{2n}$ with $\bar{g}(z)\in \Omega^{2n}$,
as the concatenation of the following three paths in $X_n(M)$;
\begin{itemize}
    \item $t\in [0,1/3]\mapsto (s_{\bar{z}_iz_i}(3t))_{i=1,2,\cdots,n}\in X_n(M)$;
    \item $t\in [1/3,2/3]\mapsto (g_{3t-1}(z_i))_{i=1,2,\cdots,n}\in X_n(M)$;
    \item $t\in [2/3,1]\mapsto (s_{g(z_i)\bar{z}_i}(3t-2))_{i=1,2,\cdots,n} \in X_n(M)$.
\end{itemize}
for some isotopy $g_t$ from $id_M$ to $g$.
\begin{rmk}
By \refthm{contractible}, $\Diff_\omega(M,\partial M)_0$  is simply connected and $\gamma(g,z)$ does not depend on the isotopy $g_t$. Also, observe that for each $g\in \Diff_\omega(M,\partial M)_0$, the set of points $z$ where $\gamma(g;z)$ is well-defined has full measure in $X_n(M)$.
\end{rmk}

Following \cite{GambaudoGhys04}, \cite{Ishida14} and \cite{Brandenbursky15}, we construct a homogeneous quasimorphism of $\Diff_\omega(M,\partial M)_0$ from a  homogeneous quasimorphism of  $B_2(M)$. 
Let $\varphi:B_2(M)\to \RR$ be a homogeneous quasimorphism of  $B_2(M)$.
We define a function $\cG^\circ(\varphi):\Diff_\omega(M,\partial M)_0 \to \RR$ as 
\[
\cG^\circ(\varphi)(f)=\int_{X_2(M)}\varphi(\gamma(f;z))dz
\]
and 
a function $\cG(\varphi):\Diff_\omega(M,\partial M)_0 \to \RR$ as
\[
\cG(\varphi)(f)=\lim_{p\to +\infty}\frac{\cG^\circ(\varphi)(f^p)}{p},
\] 
which is the homogenization of $\cG^\circ(\varphi)$.

Once we show that $\cG$ is a well-defined injective homomorphism from $Q(B_2(M))$ to $Q(\Diff_\omega(M,\partial M)_0)$, the infinite-dimensionality of $Q(\Diff_\omega(M,\partial M)_0)$ follows from \reflem{inftyDimBraid}.
To do this, we show that for any $f \in \Diff_\omega(M,\partial M)_0$, the function $\varphi(\gamma(f;\cdot)):X_2(M)\to \RR$, $z\mapsto \varphi(\gamma(f;z))$, is bounded, using a compactification of $X_2(M)$. 

\subsection{Well-definenss of $\cG$}

To show that $\cG$ and $\cG^\circ$ are well-defined, we make use of the following estimation of the word length of the cocycle $\gamma$. 
We postpone proving \reflem{finiteness} in \refsec{wordLen} since the proof requires a compactification technique for $X_2(M)$, which is natural, but technical.

\begin{lem}\label{Lem:finiteness}
    If $f \in \Diff_\omega(M,\partial M)_0$ and $S$ is a finite generating set of $\pi_1(X_2(M),\bar{z})$ where $\bar z \in X_2(\hat{M})$, then there is a constant $K(f, S)$ such that 
$$\ell_S(\gamma(f;z))\leq K(f, S)$$
for almost every $z$ in $\Omega^4$.
\end{lem}

Now, we show that $\cG$ and $\cG^\circ$ are well-defined.
\begin{thm}\label{Thm:wellDefineness}
    Let $\varphi$ be a homogeneous quasimorphism of  $B_2(M)$. The functions $\cG^\circ(\varphi)$ and $\cG(\varphi)$ are well-defined quasimorphisms. In particular, $\cG(\varphi)$ is homogeneous.
\end{thm}
\begin{proof}
    Let $f$ be a diffeomorphism  in $\Diff_\omega(M,\partial M)_0$.
    We claim that the integration $\cG^\circ(\varphi)(f)$ produces a well-defined real value.
    Choose a finite generating set $S$ of $P_2(M)$. 
    By \reflem{finiteness}, there is a constant $K=K(f,S)$ such that $$\ell_S(\gamma(f;z))\leq K$$
    for almost every $z$ in $\Omega^4$. 

    Now, we consider a function $\sigma:g^{-1}(\Omega^4)\cap \Omega^4 \to P_2(M)$ defined by $\sigma(z)=\gamma(\varphi;z)$.
    We show that  $\varphi\circ \sigma$ is measurable and its integration is finite. 
    Recall that $\Omega^4$ is  an open, dense, contractible subset of $X_2(M)$ which has full measure, and so is $g^{-1}(\Omega^4)$. Hence, $\sigma$ is continuous on each component of $g^{-1}(\Omega^4)\cap \Omega^4$, namely, $\sigma$ is continuous at almost every $z$.
    Then, since \[\{\varphi(g):g\in P_2(M) \text{ and }\ell_S(g)\leq K\}\]
    is a finite subset of $\RR$,
     $\varphi\circ \sigma:X_2(M)\to \RR$ is an essentially bounded function.
    Here, the value of $\varphi\circ \sigma$ at a point in the complement of $g^{-1}(\Omega^4)\cap \Omega^4$ is assigned arbitrarily. Since  $g^{-1}(\Omega^4)\cap \Omega^4$ has full measure, the assignment is not significant.
    Therefore, we can see that $\varphi\circ \sigma$ is measurable and the integration is finite.

    The remaining part is to show that $\cG^\circ(\varphi)$ and $\cG(\varphi)$ satisfy the quasimorphism condition and $\cG(\varphi)$ is homogeneous.
    This part can be done by standard computations, using the fact that $\varphi$ is a homogeneous quasimorphism.
\end{proof}

\subsection{$\xi$-supported Dehn twists and slidings}

One way to prove \refthm{injectivity} is to show that for any non-trivial quasimorphism $\varphi$ in $Q(B_2(M))$, $\cG(\varphi)$ is non-trivial, that is, $\cG(\varphi)(g)\neq 0$ for some $g\in \Diff_\omega(M,\partial M)$.
To construct such a $g$, we first choose a pure braid $\beta\in P_2(M)$ such that $\varphi(\beta)\neq 0$. 
Then, we construct an element $g\in \Diff_\omega(M,\partial M)$ such that for any $z$ in some region $D$ of $X_2(M)$ with large area, $\gamma(g;z)$ is conjugate to $\beta$ in $B_2(M)$ and so $\cG(\varphi)(g)\neq 0$.

One necessary property for showing that  $\cG(\varphi)(g)\neq 0$ is that the $\varphi$-values of $\gamma(g;z), z\nin D$ have a small contribution to the value $\cG(\varphi)(g)$.
To find such an element $g$, from now on, we introduce a specific construction of a sequence $\{\tau_i\}_{i\in \NN}$ of density-preserving representatives of the Dehn twist in $\Mod(M, \{\bar{z}_1,\bar{z}_2\})=B_2(M)$ along a given two-sided curve.

After that, given a finite c-generating set $S$ of $B_2(M)$, we show that there is a $K>0$ such that
\[q_S(\gamma(\tau_i;z))\leq K\] 
for almost every $z\in X_2(M)$ and any $i\in\NN$,
where $q_S$ is the norm c-generated by $S$.
This implies the desired property by \refprop{qValue}.
Unfortunately, this is not directly implied by \reflem{finiteness}. 
In the end, we prove \reflem{canonicalDehnTwist}.

We first introduce sliding isotopies on a M\"obius band and a disk as toy models for the desired Dehn twists and their associated isotopies.
For any positive numbers $w, d$ with $d<w$, we say that a smooth function $f:[0,w]\to \RR$ is a \emph{$(w,d)$-step function} if  
\begin{itemize}
    \item $f=1$ on $[0,w-d]$;
    \item $f$ is strictly decreasing on $[w-d, w-d/2]$;
    \item $f$ vanishes on $[w-d/2,w]$.  
\end{itemize}
Likewise, we say that a smooth function $f:[-w,w]\to \RR$ is a \emph{$(w,d)$-bump function} if the restriction of $f$ onto $[0,w]$ is a $(w,d)$-step function and $f$ is an even function.

\begin{const}[Sliding isotopy]\label{Const:sliding}
    Let $a, d$ be positive numbers such that $a<1$ and  $d<a/4$.     
    Set $S_a$ as a closed M\"obius band $M_{a}$ with width $a$ or a closed Euclidean disk with radius $\sqrt{a/\pi}$. 
    Note that the area of $S_a$ is $a$. We construct a \emph{$(a,d)$-sliding isotopy $\chi_t$} on $S_a$ as follows.
    
    If $S_a$ is $M_{a}$, then we take a $(a/2,d)$-bump function $\frak{b}:[-a/2,a/2]\to \RR$.
    Define an isotopy $\widetilde{\chi}_t:\RR\times [-a/2,a/2]\to \RR\times [-a/2,a/2], \ t\in [0,1]$ on the universal cover of $M_{a}$ as 
    $\widetilde{\chi}_t(x,y)=(x+t\frak{b}(y), y), \ t\in [0,1]$.
    Then, the isotopy $\widetilde{\chi}_t$ induces an isotopy $\chi_t$ on $S_a$ such that  $\chi_t=\pi\circ \widetilde{\chi}_t$.

    If $S_a$ is a disk with radius $\sqrt{a/\pi}$, we take a $(\sqrt{a/\pi},d)$-step function $\frak{s}:[0,\sqrt{a/\pi}]\to \RR$. Then, we define an isotopy $\chi_t: S_a\to S_a, \ t\in [0,1]$ as $\chi_t(r, \theta)=(r,\theta+2\pi t\frak{s}(r)), \ t\in [0,1]$ under the standard polar coordinate $(r,\theta)$.
    
    From the construction, we can see that a $(a,d)$-sliding isotopy $\chi_t$ on $S_a$ is an isotopy of density-preserving diffeomorphisms on $S_a$, fixing a $d/2$-neighborhood of the boundary $\partial S_a$.        
\end{const}

\begin{figure}[ht]
        \centering
\includegraphics[width=0.5\linewidth]{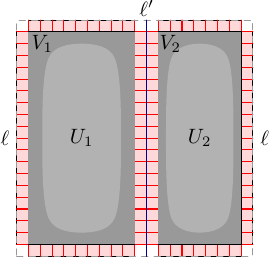}
        \caption{A decomposition into foliated strips and junctions. The red rectangles are strips, foliated by the red geodesic arcs. The white rectangles are junctions which are adjacent to exactly three strips.}
        \label{Fig:decomposition}
\end{figure}
Then, to specify a representative of the curve for a Dehn twist, we set a decomposition of $M$ as follows.
\begin{conv}\label{Conv:dividingLine}
    From now on, we think of $\hat{M}$ as the set $(-1/2,1/2)\times I$, which is a component of $\pi^{-1}(\hat{M})$.
    Given two positive numbers $a_1$, $a_2$  with $a_1+a_2\leq 1$, we set \[a_1^*=\frac{a_1}{a_1+a_2} \text{ and }a_2^*=\frac{a_2}{a_1+a_2}.\]
    Then, we denote by $\ell(a_1,a_2)$ the vertical line $\{-1/2+a_1^*\}\times I$ in $\hat{M}$.
    For instance, in \reffig{decomposition}, $\ell'$ divides $\hat{M}$ into two subrectangles.
    When $\ell'=\ell(a_1,a_2)$, the areas of left and right rectangles are $a_1^*$ and $a_2^*$, respectively.
    Moreover, for any $\epsilon$ with $0<\epsilon<a_i^*/2, \ i=1,2,$ 
    we also set 
    \[
    V_1(a_1,a_2,\epsilon)=[-\frac{1}{2}+\epsilon,-\frac{1}{2}+a_1^*-\epsilon]\times[-\frac{1}{2}+\epsilon,\frac{1}{2}-\epsilon] \text{ and } V_2(a_1,a_2,\epsilon)=[-\frac{1}{2}+a_1^*+\epsilon,\frac{1}{2}-\epsilon]\times[-\frac{1}{2}+\epsilon,\frac{1}{2}-\epsilon].
    \]
    Note that $M\setminus \big( V_1(a_1,a_2,\epsilon)\cup V_2(a_1,a_2,\epsilon) \big)$ is the $\epsilon$-neighborhood of $\partial M \cup \ell \cup \ell(a_1,a_2)$. 
    See \reffig{decomposition}.
    We write $N(a_1,a_2,\epsilon)=M\setminus \big( V_1(a_1,a_2,\epsilon)\cup V_2(a_1,a_2,\epsilon) \big)$. 
    
    Now, we decompose $N(a_1,a_2,\epsilon)$ into ten subrectangles as in \reffig{decomposition}: four white (closed) rectangles; six red (open) rectangles.
    In particular, we foliate six red rectangles by horizontal or vertical (geodesic) arcs as in \reffig{decomposition}. We call each red rectangle with such a foliation a \emph{strip} of $N(a_1,a_2,\epsilon)$ and each white rectangle a \emph{junction} of $N(a_1,a_2,\epsilon)$.
    Note that each junction is adjacent to three strips.
\end{conv}
\begin{rmk}
   We follow the notions of strips and junctions as introduced in \cite{BestvinaHandel95}. However, our definitions are slightly different from those in \cite{BestvinaHandel95}.
\end{rmk}

Given a two-sided simple closed curve $\gamma$ in $M$, we choose a `special' representative of $\gamma$ with respect to junctions and foliations of strips to construct desired representatives of the Dehn twist along $\gamma$.
In the following construction, we detail how to choose a representative of $\gamma$.
Recall that in a surface $S$, we say that two embedded $1$-manifolds $\alpha$ and $\beta$ in $S$ \emph{bound a bigon} in a subset $A$ of $S$ if there is an embedded bigon $P$ in $A$ such that 
$P\cap (\alpha\cup \beta)=\partial P$, one side of $P$ is a subarc of $\alpha$ and the other side of $P$ is a subarc of $\beta$.

\begin{const}[Minimal position]\label{Const:minimalCurve}
    Let $a_1, a_2$ be two positive numbers with $a_1+a_2\leq 1$. Choose $\epsilon$ with $0<\epsilon<a_i^*/4,\ i=1,2$. 
    Set \[\ell'=\ell(a_1,a_2), V_1=V_1(a_1,a_2,\epsilon), V_2=V_2(a_1,a_2,\epsilon) \text{ and } N=N(a_1,a_2,\epsilon), \]
    introduced in \refconv{dividingLine}.
    Also, see \reffig{decomposition}.
    Assume that $\bar{z}_i\in \interior(V_i)$ for all $i=1,2$. 
    
    Let $\gamma$ be a two-sided simple closed curve in $M\setminus\{\bar{z}_1,\bar{z}_2\}$ that is either peripheral or generic.
    We say that a representative $c$ of $\gamma$ \emph{is in minimal position with respect to $N$} if it satisfies the following conditions:
    \begin{itemize}
        \item $c$ is a smooth curve in $N$;
        \item $c$ intersects perpendicularly $\ell$ and
        leaves of the foliations of strips of $N$;
        \item  $c$ and $\ell$ are in  minimal position in $N$, that is, 
        $c$ and $\ell$ do not bound a bigon in $N$;
        \item for any side $s$ of a junction $J$ of $N$, $c$ and $s$ do not bound a bigon in $J$ (see \reffig{junction});
        \item(monotone condition) for any junction $J$ of $N$ and any component $d$ of $c \cap \closure{J}$, there is a pair of monotone smooth maps $\delta_1,\delta_2:[0,1]\to \RR$ such that the path $\delta:[0,1]\to M$, defined as $\delta(t)=\pi(\delta_1(t),\delta_2(t))$, is a regular parametrization of $d$ (see \reffig{junction}).
    \end{itemize}
     We call a closed subarc $d$ of $c$ a \emph{branch} of $c$ if $d$ is a component of the intersection of $c$ with a junction or a strip. 
     For a junction or strip $R$ of $N$, $\nn(c,R)$ denotes the number of branches of $c$ contained in $R$. 
 
    We can always find a representative of $\gamma$ in minimal position with respect to $N$.
    To see this, first take a smooth representative of $\gamma$ such that it is contained in $N$ and intersects perpendicularly to $\ell$ and the leaves of the foliations of strips of $N$.   
    By the standard technique of removing a bigon, we can take a desired representative of $\gamma$.

\end{const}
  \begin{figure}[ht]
        \centering
\includegraphics[width=0.5\linewidth]{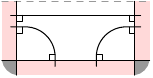}
        \caption{Possible subarcs in a junction.}
        \label{Fig:junction}
\end{figure}

    Note that if $c$ is in minimal position with respect to $N$, then $c$ intersects $\ell$ only at junctions.
    Also, since $c$ is two-sided, and peripheral or generic in $M\setminus\{\bar{z}_1,\bar{z}_2\}$, $c$ bounds either a closed M\"obius band or a closed disk in $M$. 
    We denote by $S_{in}(c)$ the surface bounded by $c$, that is, $S_{in}(c)$ is either a M\"obius band or a disk with $\partial S_{in}(c)=c$.
    Then, the closure of the complement of $S_{in}(c)$ is homeomorphic to a closed annulus or a real projective plane with two open disks removed and we denote it by $S_{out}(c)$.
    Note that $M=S_{in}(c)\cup S_{out}(c)$ and $S_{in}(c)\cap S_{out}(c)=c$.

    Now, given a curve $c$ in minimal position with respect to $N$, we construct a representative of the Dehn twist along $\gamma$ in a specific way.
\begin{const}[$\xi$-supported Dehn twist/sliding isotopy]\label{Const:supportedTwist}
    Let $a_1, a_2$ be two positive numbers with $a_1+a_2\leq 1$. Choose $\epsilon$ with $0<\epsilon<a_i^*/4,\ i=1,2$. 
    Set \[\ell'=\ell(a_1,a_2), V_1=V_1(a_1,a_2,\epsilon), V_2=V_2(a_1,a_2,\epsilon) \text{ and } N=N(a_1,a_2,\epsilon), \]
    introduced in \refconv{dividingLine}.
    Also, see \reffig{decomposition}.
    Assume that $\bar{z}_i\in \interior(V_i)$ for all $i=1,2$.
    
    Let $\gamma$ be a two-sided simple closed curve in $M\setminus\{\bar{z}_1,\bar{z}_2\}$ that is either peripheral or generic. 
    Fix a representative $c$ of $\gamma$ that is in minimal position with respect to $N$ (\refconst{minimalCurve}).
    By the tubular neighborhood theorem, for any sufficiently small number $\xi>0$, the open $\xi$-neighborhood $N_\xi(c)$ of $c$ is an embedded annulus in $N$. Fix such a $\xi>0$.

    Let $\theta:S_a\to S_{in}(c)$ be a density-preserving embedding, and let $\chi_t$ be a $(a,d)$-sliding isotopy given by \refconst{sliding},  where $a$ is the area of $S_{in}(c)$ and $S_a$ is a surface homeomorphic to $S_{in}$ with the same area, following the notation in \refconst{sliding}.
    We say that the pair $(\theta, \chi_t)$ is \emph{$\xi$-supported} if the $\theta$-image of the closed $d$-neighborhood of $\partial S_a$ is contained in $N_\xi(c)$.
    Note that given $\theta$ and $\xi$, for any sufficiently small $d>0$, $(\theta, \chi_t)$ is $\xi$-supported since $\theta$ is smooth.

    Given a $\xi$-supported pair $(\theta, \chi_t)$, we construct an isotopy $\tau_t$ in $\Diff_\omega(M,\partial M)_0$ as follows:
    consider the isotopy $\theta \circ \chi_t \circ \theta^{-1}$.  Since the isotopy is supported in $\interior\big(S_{in}(c)\big)$, we can extend it by the identity on $S_{out}(c)$.
    Then, we obtain an isotopy $\tau_t$ in $M$ such that $\tau_t=\theta \circ \chi_t \circ \theta^{-1}$ on $S_{in}(c)$ and $\tau_t=id$ on $S_{out}(c)$.
    From the construction, $\tau_t\in \Diff_\omega(M,\partial M)_0$.
    We call $\tau_t$ a \emph{$\xi$-supported sliding isotopy associated with $(\theta, \chi_t)$}.
    In particular, $\tau_1$ is called a \emph{$\xi$-supported Dehn twist along $c$}.
    In fact, $\tau_1$ is a density-preserving representative  of a Dehn twist $T_\gamma$ (in $M\setminus \{\bar{z}_1,\bar{z}_2\}$) along $\gamma$ since $\tau_1(\bar{z}_i)=\bar{z}_i, \ i=1,2$.\qedhere
\end{const}

\subsection{Auxiliary braids in $P_2(M)$}\label{Sec:auxBraid}
\begin{figure}
\begin{subfigure}{.4\textwidth}
  \centering
  \includegraphics[width=\linewidth]{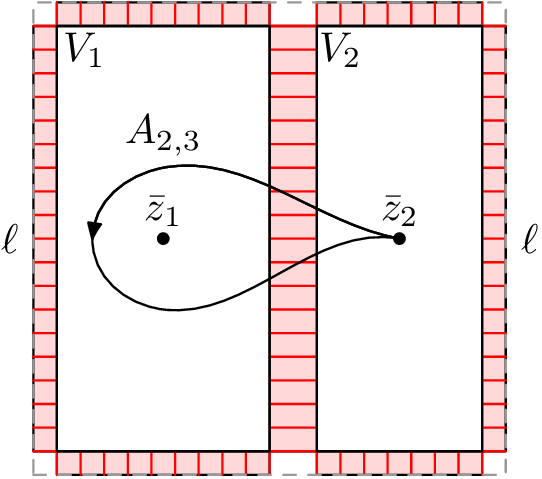}
  \caption{}
  \label{Fig:A23}
\end{subfigure}
\begin{subfigure}{.4\textwidth}
  \centering
  \includegraphics[width=\linewidth]{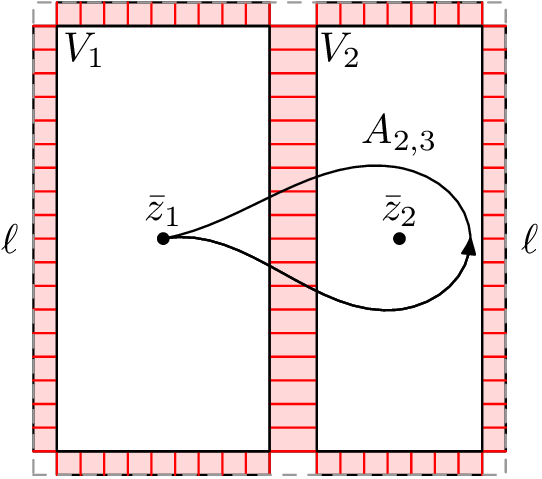}
  \caption{}
  \label{Fig:A23Alt}
\end{subfigure}
\begin{subfigure}{.4\textwidth}
  \centering
  \includegraphics[width=\linewidth]{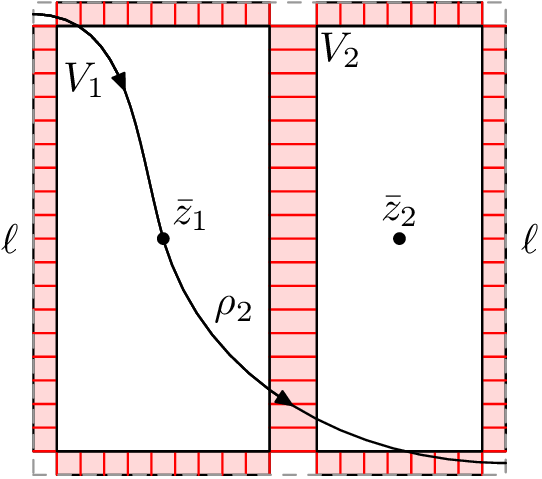}
  \caption{}
  \label{Fig:R2}
\end{subfigure}
\begin{subfigure}{.4\textwidth}
  \centering
  \includegraphics[width=\linewidth]{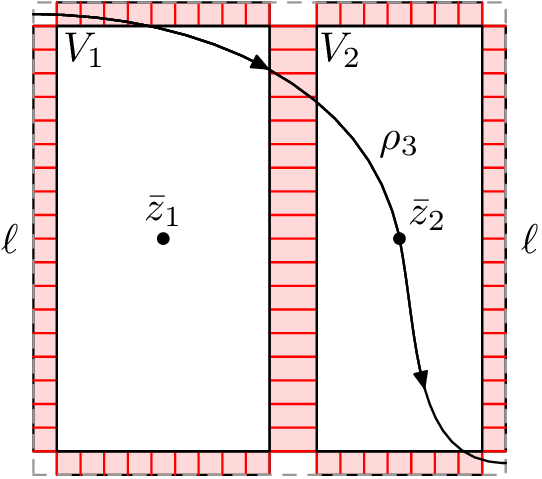}
  \caption{}
  \label{Fig:R3}
\end{subfigure}
\caption{A finite generating set of $P_2(M)$. In each figure, one of $\bar{z}_i$ does not move and the other moves along the indicated path. In particular, the braids of (A) and (B) represent the same braid, denoted by $A_{2,3}$. The braids of (C) and (D) are denoted by $\rho_2$ and $\rho_3$, respectively. See \cite[Proposition~11]{GoncalvesGuaschi17} and compare this with \cite[Figure~1]{GoncalvesGuaschi17}.} 
\label{Fig:genP2}
\end{figure}

In the proof of \reflem{canonicalDehnTwist}, given a $\xi$-supported Dehn twist $\tau$ and a finite c-generating set $S$ of $B_2(M)$, we need to find an upper bound of the norm $q_S(\gamma(\tau,z))$ of the braids $\gamma(\tau,z)$.
To find such a bound, we factorize  the braids $\gamma(\tau,z)$ in some canonical way. 
In this section, we introduce two classes of auxiliary braids which are useful to factorize the braids $\gamma(\tau,z)$.

First, recall that $P_2(M)=\Gamma_{2,1}
(\RP^2)=P_2(\RP\setminus \{x_1\})$ where $x_1$ is corresponding to $\partial M$. 
It is known by \cite[Proposition~11]{GoncalvesGuaschi17} that $P_2(M)$ is finitely presented.
Recall that $(\bar{z}_1,\bar{z}_2)$ is the base point for $P_2(M)$ and each braid in $P_2(M)$ is presented as a pair of trajectories of $\bar{z}_i$. 
\reffig{genP2} describes three generators $A_{2,3}, \rho_2 \text{ and }\rho_3$ of $P_2(M)$, that is, $P_2(M)=\langle A_{2,3}, \rho_2, \rho_3 \rangle$.
For instance, in \reffig{A23}, $\bar{z}_2$ moves around $\bar{z}_1$ counter-clockwise while $\bar{z}_1$ goes nowhere, that is, the trajectory of $\bar{z}_1$ is a constant path.
Observe that \reffig{A23Alt} also represents  $A_{2,3}$.

    \begin{rmk}\label{Rmk:stdGen}
    In fact,  \cite[Proposition~11]{GoncalvesGuaschi17} says that $P_2(M)$ is generated by the five braids,\[A_{1,2}, 
    \ A_{1,3}, \ A_{2,3}, \ \rho_2 \text{ and } \rho_3.\]
    In \cite[Proposition~11]{GoncalvesGuaschi17}, they think of $\Gamma_{2,1}(\RP^2)$ as a subgroup of $P_3(\RP^2)$ with base point $(x_1,x_2,x_3)$, fixing $x_1$. 
    Comparing our convention, $x_2,x_3$ are corresponding to $\bar{z}_1,\bar{z}_2$, respectively.
    By simple computations or \cite[Proposition~11]{GoncalvesGuaschi17}, we can obtain the following relations:
    \[
    A_{1,2}=\rho_2^{-1}A_{2,3}\rho_2^{-1} \text{ and }A_{1,3}=\rho_3^{-2}A_{2,3}^{-1}.
    \]
    These imply that $ A_{2,3}, \ \rho_2 \text{ and } \rho_3$ are enough to generate $P_2(M)$.
\end{rmk}
\begin{figure}[ht]
        \centering
\includegraphics[width=0.4\linewidth]{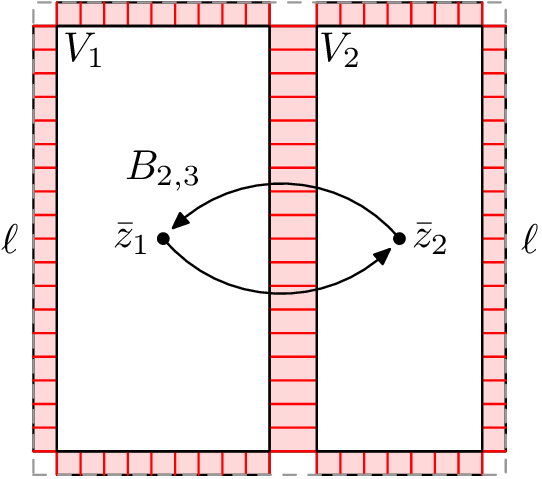}
        \caption{The braid $B_{2,3}$. It exchanges the positions of $\bar{z}_i$, twisting the strands in the counter-clockwise direction. }
        \label{Fig:B23}
\end{figure}

Note that $P_2(M)$ is an index two subgroup of $B_2(M)$.
Hence, we can see easily that $B_2(M)$ is generated by $\{B_{2,3},\rho_2,\rho_3\}$ where $B_{2,3}$ is the braid described in \reffig{B23}.
Obviously, $A_{2,3}=B_{2,3}^2$ and $\rho_2$ is conjugate to $\rho_3$ in $B_2(M)$.

Assume that $\bar{z}_i=\pi(t_i,0)$ with $-1/2< t_1 < t_2<1/2$ and $q$ is a point in $\ell$ with $q=\pi(\pm 1/2, \pm r_q)$ for some $r_q\in I$.
For each $i\in \{1,2\}$, we define geodesic paths $s^\pm_{\bar{z}_iq}:[0,1]\to M$ as
the $\pi$-image of the geodesic paths in $\widetilde{M}$ from $(t_i,0)$ to $(\pm1/2, \pm r_q)$ with respect to the sign. 
Also, we define $s^\pm_{q\bar{z}_i}$ by reversing the orientation of $s^\pm_{\bar{z}_iq}$, respectively, that is,
$s^\pm_{q\bar{z}_i}(t)=s^\pm_{\bar{z}_iq}(1-t),\ t\in [0,1]$.

For almost every $z\in \hat{M}$, we can define a braid $\eta_1(q,z)$ as the braid represented by
\[
\eta_1(q,z)(t)=
\begin{dcases}
(s_{\bar{z}_1 q}^+(2t),s_{\bar{z}_2 z}(2t))&\text{for }t\in[0,1/2]\\
(s_{q\bar{z}_1}^-(2t-1),s_{z\bar{z}_2}(2t-1))&\text{for }t\in[1/2,1].
\end{dcases}
\]
Similarly, a braid $\eta_2(q,z)$ is defined as the braid represented as
\[
\eta_2(q,z)(t)=
\begin{dcases}
(s_{\bar{z}_1z}(2t),s_{\bar{z}_2q}^+(2t))&\text{for }t\in[0,1/2]\\
(s_{z\bar{z}_1}(2t-1),s_{q\bar{z}_2}^-(2t-1))&\text{for }t\in[1/2,1].
\end{dcases}
\] 
By a simple computation, we can obtain the following lemma.
\begin{lem}\label{Lem:auxBraid}
    Assume that $\bar{z}_i=\pi(t_i,0)$ with $-1/2< t_1 < t_2<1/2$. 
   For almost every $(q,z)\in \ell \times \hat{M}$, $\eta_i(q,z)$ are well-defined. Moreover, for each $i\in \{1,2\}$, there are only four possible braids of $\eta_i(q,z)$: $\eta_{i,j},\ j=1,2,3,4$ as described in \reffig{auxBraid}.
   In particular, each $\eta_{i,j}$ is conjugate to $\eta_{1,1}$ or $\eta_{1,2}$ in $B_2(M)$ and for each $i\in\{1,2\}$, $\eta_{i,3}$ and $\eta_{i,4}$ are conjugate to $\eta_{i,1}$ and $\eta_{i,2}$, respectively,  in $P_2(M)$.
\end{lem}
\begin{rmk}\label{Rmk:etaTerm}
We have that 
    $\eta_{1,2}=\rho_2 \text{ and }  \eta_{1,1}=A_{2,3}^{-1}\cdot \rho_2$.
\end{rmk}

\begin{figure}
\begin{subfigure}{.4\textwidth}
  \centering
  \includegraphics[width=\linewidth]{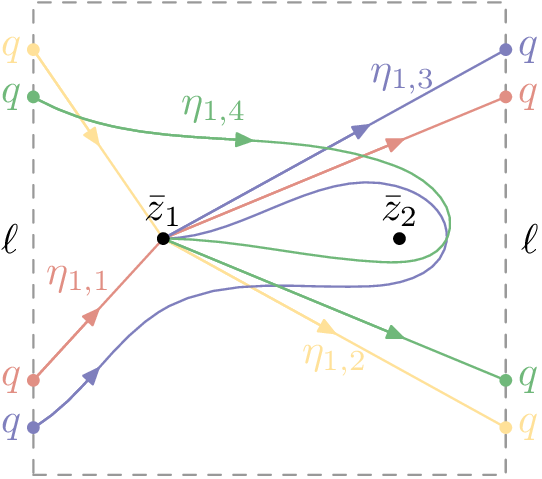}
  \caption{possible braids of $\eta_1(q,z)$}
  \label{Fig:Eta1}
\end{subfigure}
\begin{subfigure}{.4\textwidth}
  \centering
  \includegraphics[width=\linewidth]{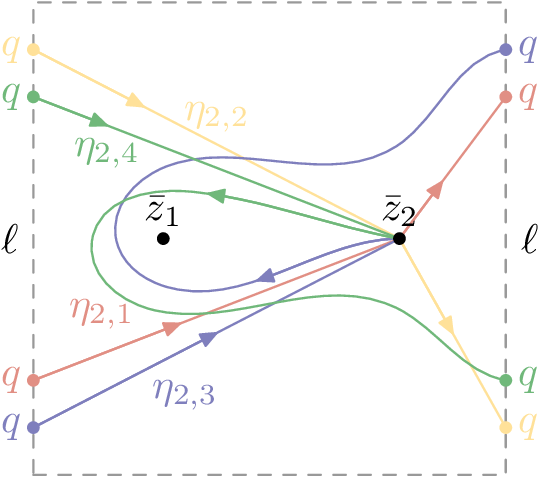}
  \caption{possible braids of $\eta_2(q,z)$}
  \label{Fig:Eta2}
\end{subfigure}

\caption{Each braid $\eta_{i,j}$ is represented as a pair of  two trajectories: one constant path at $\bar{z}_{3-i}$; one non-trivial trajectory of $\bar{z}_i$.}
\label{Fig:auxBraid}
\end{figure}

Let $r$ be an embedded arc in $M$ that intersects $\ell$ only at a unique end point $e$ of $r$.
Now, we think of $\hat{M}$ as $(-1/2,1/2)\times I$.
Then, $r$ can be uniquely lifted onto $[-1/2,1/2]\times I$ and so there is a unique lifting $\tilde{e}$ of $e$ in the lifting of $r$. 
We say that $r$ \emph{is transverse to the left side (resp. right side) of $\ell$} if $\tilde{e}\in 1/2\times I$ (resp. $\tilde{e}\in -1/2\times I$).
Also, we say that $r$ \emph{is transverse to $\ell$} if it is transverse to the left side or right side of $\ell$.

We call a continuous map $\gamma:[0,1]\to M$ a \emph{regular curve} in $M$ if it is one of the following:
    \begin{itemize}
        \item a constant path in $\hat{M}$;
        \item an embedded path or simple closed curve such that for a sufficiently small $\delta>0$, each of $r_1=\gamma([0,\delta])$ and  $r_2=\gamma([1-\delta,1])$ is either transverse to $\ell$ with $r_i\cap \ell=\{\gamma(i-1)\}$ or does not intersect $\ell$.
\end{itemize}
We call $r_1$ and $r_2$   \emph{starting  and ending arcs} of $\gamma$, respectively, if $\gamma$ is not a constant path.
Note that starting  and ending arcs are not uniquely determined.

\begin{const}[Closing a curve]\label{Const:closingStrand}
    Let $\gamma:[0,1]\to M$ be a regular curve in $M$.
    For each point $q$ in $\hat{M}$, the \emph{closing} $c(\gamma,q)$ of $\gamma$ at $q$ is the curve defined as follows:
    say $w_1=\gamma(0)$ and $w_2=\gamma(1)$.
    If $\gamma$ is a constant path in $\hat{M}$, that is, $w_1=w_2$, then we define 
    \[
    c(\gamma,q)(t)= \begin{dcases}
   s_{q,w_1}(3t) & \text{if $t\in [0,1/3]$},\\
   w_1 & \text{if $t\in [1/3,2/3]$},\\
   s_{w_1,q}(3t-2)& \text{if $t\in [2/3,1]$}.\\
    \end{dcases}
    \]
    In this case, we simply denote $c(\gamma,q)$ by $c(w_1,q)$.
    Otherwise, say that $r_1$ and $r_2$ are starting  and ending arcs of $\gamma$, respectively. Then, we define
    \[
    c(\gamma,q)(t)= \begin{dcases}
   s_{q,w_1}^\circ(3t) & \text{if $t\in [0,1/3]$},\\
   \gamma(3t-1) & \text{if $t\in [1/3,2/3]$},\\
   s_{w_2,q}^\circ(3t-2)& \text{if $t\in [2/3,1]$}.\\
    \end{dcases}
    \]
    where for each $i\in \{1,2\}$,  $s_{q,w_i}^\circ:[0,1]\to M$ is defined as 
    \[
    s_{q,w_i}^\circ=
    \begin{dcases}
   s^-_{q,w_i} & \text{if $r_i$ is transverse to the right side of $\ell$},\\
   s_{q,w_i} & \text{if }w_i\in \hat{M},\\
   s^+_{q,w_i}& \text{if $r_i$ is transverse to the left side of $\ell$}.\\
    \end{dcases}
    \]
    and $s_{w_i,q}^\circ(t)=s_{q,w_i}^\circ(1-t)$ for $t\in [ 0,1]$.
\end{const}

\begin{const}\label{Const:oneOrbitBraid}
    Let $\gamma:[0,1]\to M$ be a regular curve in $M$.
    For almost every $u\in \hat{M}$, we can define $\beta_1(\gamma,u)$ as the braid represented as 
    \[
    \beta_1(\gamma,u)(t)=(c(\gamma,\bar{z}_1)(t),c(u,\bar{z}_{2})(t)), \ t\in [0,1].
    \]
    Likewise, we can define $\beta_2(\gamma,u)$ as the braid represented as 
    \[
    \beta_2(\gamma,u)(t)=(c(u,\bar{z}_{1})(t),c(\gamma,\bar{z}_2)(t)), \ t\in [0,1]
    \]
    for almost every $u\in \hat{M}$.
\end{const}
\begin{rmk}\label{Rmk:diskBraid}
    When $\gamma((0,1))\subset \hat{M}$, we can think of $\beta_i(\gamma,u)$ as pure braids with two strands in the disk $\hat{M}$, namely, $\beta_i(\gamma,u)=A_{2,3}^n$ for some $n\in \ZZ$. See \reffig{A23} and \reffig{A23Alt}.
\end{rmk}

\begin{lem}\label{Lem:moveOneStrand}
    Let $\gamma_1,\gamma_2:[0,1]\to M$ be simple closed curves with $\gamma_i(0)=\gamma_i(1)\in \hat{M}$ and $u$ a point in $\hat{M}$.
    Assume that there is a free homotopy $H:[0,1]\times [0,1] \to M$ satisfying the following:
    \begin{itemize}
        \item $H(0,t)=\gamma_1(t)$ and $H(1,t)=\gamma_2(t)$;
        \item $H(s,0)=H(s,1)$ for all $s\in [0,1]$;
        \item $u\notin H([0,1]\times [0,1])$.
    \end{itemize}
     If $\beta_i(\gamma_j,u)$ are well-defined for some $i\in\{1,2\}$, then $\beta_i(\gamma_1,u)$ is conjugate to $\beta_i(\gamma_2,u)$ in $P_2(M)$.
\end{lem}
\begin{proof}
    We just provide the proof of the case of $i=1$ since the case of $i=2$ can be shown in the exactly same way. 
    Assume that $i=1$. Say that $\alpha:[0,1]\to M$ is the path given by $\alpha(s)=H(s,0)$ and $\gamma_i(0)=v_i$. 
    Note that $\alpha$ is a path in $M\setminus\{u\}$, joining $v_1$ with $v_2$ by the third condition of $H$.

    First, observe that $\beta_1(\gamma_1,u)$ can be represented as 
    \[
    c_1(t)=
    \begin{dcases}
   (s_{\bar{z}_1,v_1}(3t), s_{\bar{z}_2,u}(3t)) & \text{if $t\in [0,1/3]$},\\
    (\tilde{\gamma}_1(3t-1),u) & \text{if $t\in [1/3,2/3]$},\\
    (s_{v_1,\bar{z}_1}(3t-2), s_{u,\bar{z}_2}(3t-2))& \text{if $t\in [2/3,1]$},\\
    \end{dcases}
    \]
    where 
    \[
\tilde{\gamma}_1(t)=\begin{dcases}
   \alpha(3t) & \text{if $t\in [0,1/3]$},\\
   \gamma_2(3t-1) & \text{if $t\in [1/3,2/3]$},\\
   \alpha(-3t+3)& \text{if $t\in [2/3,1]$}.\\
    \end{dcases}
    \]
On the other hands,  $\beta_1(\gamma_2,u)$ can be represented as 
    \[
    c_2(t)=
    \begin{dcases}
   (s_{\bar{z}_1,v_2}(3t), s_{\bar{z}_2,u}(3t)) & \text{if $t\in [0,1/3]$},\\
    (\tilde{\gamma}_2(3t-1),u) & \text{if $t\in [1/3,2/3]$},\\
    (s_{v_2,\bar{z}_1}(3t-2), s_{u,\bar{z}_2}(3t-2))& \text{if $t\in [2/3,1]$},\\
    \end{dcases}
    \]
    where 
    \[
\tilde{\gamma}_2(t)=\begin{dcases}
   v_2 & \text{if $t\in [0,1/3]$},\\
   \gamma_2(3t-1) & \text{if $t\in [1/3,2/3]$},\\
   v_2& \text{if $t\in [2/3,1]$}.\\
    \end{dcases}
    \]
Now, we define $\delta$ as the pure braid represented as   
\[
d(t)=\begin{dcases}
   (s_{\bar{z}_1,v_2}(4t), s_{\bar{z}_2,u}(4t)) & \text{if $t\in [0,1/4]$},\\
   (v_2, u) & \text{if $t\in [1/4,1/2]$},\\
    (\alpha(-4t+3),u) & \text{if $t\in [1/2,3/4]$},\\
    (s_{v_1,\bar{z}_1}(4t-3), s_{u,\bar{z}_2}(4t-3))& \text{if $t\in [3/4,1]$}.
    \end{dcases}
\]
Then, it follows from the above representations $c_i(t),d(t)$
that \[
\beta_1(\gamma_2,u)=\delta  \beta_1(\gamma_1,u) \delta^{-1}.
\]
Thus, we are done.
\end{proof}

In a similar way, we can also prove the following lemma.
\begin{lem}\label{Lem:basePointChange}
    Let $\gamma:[0,1]\to M$ be a simple closed curves with $\gamma(0)=\gamma(1)\in \hat{M}$ and $u,v$ two points in $\hat{M}$.
    Assume that there is a path $\alpha:[0,1]\to M$ such that $\alpha(0)=u$, $\alpha(1)=v$ and $\alpha([0,1])\cap \gamma([0,1])=\emptyset$.
    If $\beta_i(\gamma,u)$ and $\beta_i(\gamma,v)$ are well-defined for some $i\in \{1,2\}$, then $\beta_i(\gamma,u)$ is conjugate to  $\beta_i(\gamma,v)$ in $P_2(M)$. 
\end{lem}

Finally, we end this subsection by showing that $\beta_i(\gamma,u)$ can be factorized in a canonical way.

\begin{lem}\label{Lem:canonicalFactor}
    Let $\gamma:[0,1]\to M$ be an embedded path or closed curve with $\gamma(0),\gamma(1)\in \hat{M}$ and $c=\gamma([0,1])\subset M\setminus\{\bar{z}_1,\bar{z}_2\}$. 
    Assume that $\bar{z}_i=\pi(t_i,0)$ with $-1/2< t_1 < t_2<1/2$ and  that $c$ and $\ell$ are transverse.  
    If $\fN=|c \cap \ell|\neq 0$, that is, \[
    c\cap \ell=\{\gamma(s_1),\gamma(s_2),\dots, \gamma(s_\fN)\}
    \text{ and } 
    0=s_0<s_1<s_2<\cdots<s_\fN<s_{\fN+1}=1, 
    \]
 then, for almost every $v\in \hat{M}$, $\beta_i(\gamma,v)$ are well-defined  and 
    \[
    \beta_i(\gamma,v)= \beta_i(\gamma_0,v)\cdot \eta_i(\gamma(s_0),v)^{\epsilon_0}\cdot \beta_i(\gamma_1,v)\cdot\eta_i(\gamma(s_1),v)^{\epsilon_1}\cdot \ldots \cdot  \eta_i(\gamma(s_\fN),v)^{\epsilon_\fN}\cdot \beta_i(\gamma_\fN,v)
    \]
    for some $\epsilon_i\in\{\pm1\}$
    where  each $\gamma_i:[0,1]\to M$ is a reparametrization of $\gamma|_{[s_i,s_{i+1}]}$, preserving the orientation.
\end{lem}
\begin{proof}
    This immediately follows from the constructions of $\eta_i$ and $\beta_i$.
\end{proof}
\begin{rmk}\label{Rmk:eachFactor}
    Note that by \refrmk{diskBraid}, 
    $\beta_i(\gamma_j, v)=A_{2,3}^{n_{i,j}}$ for some $n_{i,j}\in \ZZ$ and by \reflem{auxBraid}, $\eta_i(\gamma(s_j),v)$ are conjugate to $\eta_{1,1}$ or $\eta_{1,2}$ in $B_2(M)$.
\end{rmk}

\subsection{Conjugation-generated norms for the braids associated with $\xi$-supported Dehn twists}

Now, given a two-sided simple closed curve $\kappa$ and a finite c-generating set $S$ of $B_2(M)$, we construct a sequence $\{\tau_i\}_{i\in \NN}$ of $\xi_i$-supported Dehn twists along its representatives of $\kappa$ such that $q_S(\gamma(\tau_i,z))$ are bounded by some uniform constant.
For the convenience, we denote the \emph{central curve} $\pi(\RR\times 0)$ of $M$ by $\gamma_0$. 

\begin{lem}\label{Lem:canonicalDehnTwist}
    Let $a_1,a_2,\epsilon$ be positive numbers such that $a_1+a_2\leq 1$ and $0<\epsilon<a_i^*/4$ for all $i=1,2$. 
    Assume that $\bar{z}_i\in \interior(V_i(a_1,a_2,\epsilon)) \cap \gamma_0$ for all $i=1,2$. If $S$ is a finite  c-generating set of $P_2(M)$ and $\kappa$ is a two-sided simple closed curve in $M\setminus\{\bar{z}_1,\bar{z}_2\}$ that is either peripheral or generic, then there are a constant $K$ and a sequence $\{(\epsilon_n, k_n, \xi_n, \tau_n)\}_{n\in \NN}$ satisfying the following:
        \begin{itemize}
            \item $\{\epsilon_n\}_{n\in\NN}$ is a strictly decreasing sequence of positive numbers such that $\epsilon_1<\epsilon$ and $\epsilon_n \to 0$ as $n\to \infty$;
            \item $\{k_n\}_{n\in\NN}$ is a sequence of representatives of $\kappa$ such that for each $n\in \NN$, $k_n$ is in minimal position with respect to $N(a_1,a_2,\epsilon_n)$;
            \item $\{\xi_n\}_{n\in\NN}$ is a sequence of positive numbers such that for each $n\in \NN$, $N_{\xi_n}(k_n)$ is a tubular neighborhood of $k_n$ in $N(a_1,a_2,\epsilon_n)$;
            \item for each $n\in\NN$, $\tau_n$ is a $\xi_n$-supported Dehn twist along $k_n$;
            \item the following inequality holds:
            for any $n\in \NN$,
            \[q_S(\gamma(\tau_n^{\pm1};z))\leq K\]
            for almost every $z\in \Omega^4$,
            where  $q_S$ is the norm c-generated by $S$.
        \end{itemize}
\end{lem}
\begin{proof}
Set $\epsilon_0=\epsilon$ and  $\epsilon_i=\epsilon/2^i$ for $i\in \NN$. 
For each $i\in \ZZ_{\geq0}$, we write   
\[\ell'=\ell(a_1,a_2), V_1^i=V_1(a_1,a_2,\epsilon_i), V_2^i=V_2(a_1,a_2,\epsilon_i) \text{ and } N^i=N(a_1,a_2,\epsilon_i).\] 
For each $i\in \ZZ_{\geq0}$, $M$ has a cell decomposition $\cD_i$ induced by junctions and strips of $N^i$ (\refconst{minimalCurve}) and $V_j^i,\ j=1,2$.

We first observe that there is a smooth diffeomorphism $L:M\to M$, satisfying the following:
\begin{itemize}
    \item $L$ preserves $\ell$ and $\ell'$,
    \item $L$ is a cellular map from $\cD_i$ to $\cD_{i+1}$,
    \item $L(V_j^i)=V_j^{i+1},\ j=1,2$, and 
    \item $L$ maps each leaf of strips of $N^i$ to a leaf of strips of $N^{i+1}$.
\end{itemize}
Note that such a $L$ maps strips and junctions of $N^i$ to strips and junctions of $N^{i+1}$, respectively.  

To construct such a $L$, we take  $\varphi_x$ in $\Diff_+^\infty(I)$, satisfying 
\begin{itemize}
    \item $p_x=-1/2+a_1^*$ is a unique attracting fixed point of $\varphi_x$ in $\interior(I)$; 
    \item each of $(-1/2,p_x)$ or $(p_x,1/2)$ contains exactly one fixed point and these fixed points are repelling;
    \item $\varphi_x$ maps linearly \[[-1/2,-1/2+\epsilon_0], [p_x-\epsilon_0,p_x+\epsilon_0] \text{ and }  [1/2-\epsilon_0,1/2]
\] to
\[[-1/2,-1/2+\epsilon_1], [p_x-\epsilon_1,p_x+\epsilon_1] \text{ and }  [1/2-\epsilon_1,1/2],
\] repsectively;
\item $\varphi_x$ is also linear in small neighborhoods of each of \[-1/2+\epsilon_0, \ p_x-\epsilon_0,\ p_x+\epsilon_0 \text{ and } \ 1/2-\epsilon_0.\]
\end{itemize}
Likewise, we take  $\varphi_y$ in $\Diff_+^\infty(I)$ satisfying 
\begin{itemize}
    \item there is the unique fixed point $p_y$ of $\varphi_y$ in $\interior(I)$;
    \item $p_y\in  (-1/2+\epsilon_0,1/2-\epsilon_0)$ and it is  
    repelling;
    \item $\varphi_y$ is odd;
    \item $\varphi_y$ maps linearly    \[[-1/2,-1/2+\epsilon_0]\text{ and }  [1/2-\epsilon_0,1/2]
\text{ to }[-1/2,-1/2+\epsilon_1]\text{ and }  [1/2-\epsilon_1,1/2],
\] respectively;
\item $\varphi_x$ is also linear in small neighborhoods of each of $-1/2+\epsilon_0 \text{ and } \ 1/2-\epsilon_0$.
\end{itemize}
Note that near $\epsilon_0$-neighborhoods of attracting fixed points,  $\varphi_x$ and $\varphi_y$  are  linear maps with stretch factors $1/2$ since $1/2=\epsilon_1/\epsilon_0$.

Now, we define $L_0:I\times I \to I\times I$ as $(s,t)\mapsto (\varphi_x(s),\varphi_y(t))$. 
Since $\varphi_y$ is odd and in small neighborhoods of $1/2$ and $-1/2$, $\varphi_x$ is  linear maps with stretch factor $1/2$, fixing $\pm1/2$, we can define $L:M\to M$ as $L\circ \pi=\pi\circ L_0$.
Observe that $L$ is a contracting linear map near small open neighborhood of each junction of $N^0$, fixing end points of $\ell$ and $\ell'$ and  satisfying the desired properties.

Then, we fix a smooth representative $k_0$ of $\kappa$ in minimal position with respect to $N^0$ (\refconst{minimalCurve}).
Set $k_{i+1}=L(k_i)$ for all $i\in \ZZ_{\geq0}$. It follows from the construction of $L$ that $k_{i}$ are smooth curves in minimal position with respect to $N^i$.
In particular, the monotone condition of minimal position in \refconst{minimalCurve} is preserved under the iteration of $L$ since $L$ is a contracting linear map in each junction.

From now on, we construct a sequence of supported Dehn twists along $k_i$ satisfying the desired properties.
For each $i\in \ZZ_{\geq0}$, we denote by $a_i$ the area of $S_{in}(k_i)$ and fix a density-preserving embedding $\theta_i:S_{a_i}\to S_{in}(k_i)$.
Also, for each $i\in \ZZ_{\geq0}$, we choose a positive number $\xi_i$ so that the $\xi_i$-neighborhood $N_{\xi_i}(k_i)$ is a tubular neighborhood of $k_i$ in $N^i$, foliated by the geodesic arcs perpendicular to $k_i$.
Say that $\cA_i$ is such a foliation of $N_{\xi_i}(k_i)$.

Fix $i\in \NN$. 
Then, there is $D_i>0$ such that for any $(a_i,d)$-sliding isotopy $\chi_t$ with $d\leq D_i$, given by \refconst{sliding}, $(\theta_i,\chi_t)$ is $\xi_i$-supported. 
Choose $d<D_i$ and take an $(a_i,d)$-sliding isotopy $\chi_t$. Note that   $(\theta,\chi_t)$ is $\xi_i$-supported.
Say that $\sigma_t$ is the $\xi_i$-supported sliding isotopy associated with $(\theta_i,\chi_t)$ (\refconst{sliding}).
 
Recall that $R=\{B_{2,3},\rho_2,\rho_3\}$ is a finite c-generating set of $B_2(M)$ and $q_R$ is the norm c-generated by $R$. 
To find an upper bound for $q_R(\gamma(\sigma_t,z))$ for almost every $z\in \Omega^4$, we observe that  $\gamma(\sigma_t,z)$ can be written as a finite product of the conjugations of the auxiliary braids, $\eta_i,\beta_i$, introduced in \refsec{auxBraid}.
 
Note that it follows from the construction of $\{\sigma_t\}_{t\in [0,1]}$ that $\sigma_t$ is a part of a unique topological flow $\{\delta_s\}_{s\in\RR}$,
defined as $\delta_s=\sigma_1^{\lfloor s \rfloor}\circ 
\sigma_{s-\lfloor s \rfloor}=\sigma_{s-\lfloor s \rfloor}\circ\sigma_1^{\lfloor s \rfloor}$.
Whenever we mention an orbit of $\sigma_t$, it refers to the associated orbit of $\delta_s$.
Each orbit of the flow $\delta_s$ is either a constant path or an embedded circle.

More precisely, if $S_{a_i}$ is a M\"obius band, then $S_{a_i}$ had a foliation $\cF$ by the circle $\pi(\RR\times y), \ y\in [-a/2,a/2]$.
Each leaf of $\cF$ is oriented 
by the orientation of $\RR\times y$.
Under this orientation, each sliding isotopy rotates each leaf of $\cF$ in the positive direction by at most $2\pi$.
If $S_{a_i}$ is a closed disk, then $S_{a_i}$ has a singular foliation given by the origin and the circles centered at the origin. 
Each leaf of $\cF$ is oriented counterclockwise.
Also, as in the previous case, each sliding isotopy  rotates each leaf of $\cF$ in the positive direction by at most $2\pi$.
Hence, $\sigma_t$ preserves a singular foliation on $S_{in}(k_i)$ induced by $\cF$ and fixes each point in $S_{out}(k_i)$.

It follows from the above observation that $z_1$ and $z_2$ do not lie in the same orbit for almost every $(z_1,z_2)\in \Omega^4$. 
This fact allows us to factorize $\gamma(\sigma_1; z)$ into a product of $\beta_i(\sigma^{u_i},v)$ for $i\in \{ 1,2 \}$ and for some $v\in M$ (\refconst{oneOrbitBraid}), where 
 for any $z\in M$ and any isotopy $\{f_t\}_{t\in[0,1]}$ in $\Diff_\omega(M,\partial M)_0$, we define $f^z:[0,1]\to M$ as  
 \[
 f^z(t)=f_t(z) \text{ for } t\in [0,1].
 \]

 \begin{claim}
For almost every $z=(z_1,z_2)\in \Omega^4$ such that $z_1$ and $z_2$ do not lie in the same orbit, we have \[\gamma(\sigma_1;z)=\beta_1(\sigma^{z_1},z_2)\beta_2(\sigma^{z_2},\sigma_1(z_1))=\beta_2(\sigma^{z_2},z_1)\beta_1(\sigma^{z_1},\sigma_1(z_2)).\]
In particular, $\sigma_1(z_1)$  and $z_2$ lie on different orbits, and so do $z_1$ and $\sigma_1(z_2)$
 \end{claim}
 \begin{proof}
Recall that $\gamma(\sigma_1; z)$ can be represented as the concatenation of the following three paths in $X_2(M)$;
\begin{itemize}
    \item $t\in [0,1/3]\mapsto (s_{\bar{z}_1z_1}(3t),s_{\bar{z}_2z_2}(3t))\in X_2(M)$;
    \item $t\in [1/3,2/3]\mapsto (\sigma_{3t-1}(z_1),\sigma_{3t-1}(z_2))\in X_2(M)$;
    \item $t\in [2/3,1]\mapsto (s_{\sigma_1(z_1)\bar{z}_1}(3t-2),s_{\sigma_1(z_2)\bar{z}_2}(3t-2)) \in X_2(M)$.
\end{itemize}
Therefore, since two paths $\sigma_{3t-1}(z_1),\sigma_{3t-1}(z_2),t\in [1/3,2/3]$ lie on different orbits and so they do not intersect, we can  reparameterize freely the above as 
\begin{itemize}
    \item $t\in [0,1/3]\mapsto (s_{\bar{z}_1z_1}(3t),s_{\bar{z}_2z_2}(3t))\in X_2(M)$;
    \item $t\in [1/3,1/2]\mapsto (\sigma_{6t-2}(z_1),z_2)\in X_2(M)$;
    \item $t\in [1/2,2/3]\mapsto (\sigma_{1}(z_1),\sigma_{6t-3}(z_2))\in X_2(M)$;
    \item $t\in [2/3,1]\mapsto (s_{\sigma_1(z_1)\bar{z}_1}(3t-2),s_{\sigma_1(z_2)\bar{z}_2}(3t-2)) \in X_2(M)$.
\end{itemize}
Since for almost every $(i,z)\in \{1,2\}\times \hat{M}$, $\beta_i(\sigma^w,z)$ is well-defined for almost every $w\in \hat{M}$, it follows from the above reparametrization that $\gamma(\sigma_1;z)=\beta_1(\sigma^{z_1},z_2)\beta_2(\sigma^{z_2},\sigma_1(z_1))$ for almost every $(z_1,z_2)\in \Omega^4$. 
In a similar way, we can also see that $\gamma(\sigma_1;z)=\beta_2(\sigma^{z_2},z_1)\beta_1(\sigma^{z_1},\sigma_1(z_2))$.
\end{proof}
Since \[
q_R(\gamma(\sigma_1;z))\leq q_R(\beta_2(\sigma^{z_2},z_1))+q_R(\beta_1(\sigma^{z_1},\sigma_1(z_2))),
\] finding a uniform upper bound of  $q_R(\beta_m(\sigma^{z},v))$ for almost every $(z,v)\in \Omega^4$ such that $z$ and $v$ lie on distinct orbits is enough to find an upper bound of $q_R(\gamma(\sigma_1;w))$ for almost every $w\in \Omega^4$.

\begin{case} $\sigma^z$ is a trivial path.\\
If $\beta_m(\sigma^z,v)$ is well-defined for some $m\in\{1,2\}$ and $v\in \hat{M}$, then by a simple computation, we can see that $\beta_m(\sigma^z,v)$ is the trivial braid.
Therefore, $q_R(\beta_m(\sigma^z,v))=0$.
\end{case}
\begin{case}
$\sigma^z$ is a two-sided simple closed curve.\\
Choose a regular parametrization $\kk_i:[0,1]\to M$ of $k_i$ such that  $\kk_i(0)=\kk_i(1)\in \hat{M}$ and $\theta^{-1}_i\circ \kk_i: [0,1]\to \partial S_{a_i}$ is orientation-preserving.
Fix points $w_1^i,w_2^i$  such that $w_1^i\in \interior(S_{in}(k_i))$, $w_2^i\in \interior(S_{out}(k_i))$,  and $\beta_m(\kk_i, w_n^i)$ are well-defined for all $n,m\in\{1,2\}$.
\begin{claim}
    Let $z$ be a point in $S_{in}(k_i)\cap \hat{M}$ and $\cO_z$ the orbit containing $z$. Assume that $\sigma^z$ is a two-sided simple closed curve.
    If $\beta_m(\sigma^z, v)$ is well-defined for some $v\in \hat{M}\setminus \cO_z$ and $m\in \{1,2\}$, then $\beta_m(\sigma^z, v)$ is conjugate to $\beta_m(\kk_i, w_n^i)$ for some $n\in \{1,2\}$ in $P_2(M)$.
    In particular, $q_R(\beta_m(\sigma^z, v))=q_R(\beta_m(\kk_i, w_n^i))$.
\end{claim}
\begin{proof}
    If $v\in S_{in}(\cO_z)$, then 
    $\theta_i^{-1}(\cO_z)$ and  $\partial S_{a_i}=\theta_i^{-1}(k_i)$ bound an annulus in $S_{a_i}$ which does not contains $\theta_i^{-1}(v)$.
    When $\beta_m(\kk_i, v)$ is well-defined, this implies that there is a free homotopy from $\sigma^z$ to $\kk_i$  satisfying the condition of \reflem{moveOneStrand}.
    Hence, $\beta_m(\sigma^z, v)$  is conjugate to $\beta_m(\kk_i, v)$ in $P_2(M)$.
    Since $\interior(S_{in}(k_i))$ is path-connected, it follows from  \reflem{basePointChange} that $\beta_m(\sigma^z, v)$  is conjugate to  $\beta_m(\kk_i, w_1)$ in $P_2(M)$. 
    When $\beta_m(\kk_i,v)$ is not well-defined,
    we can take another point $v'$ in a small open ball centered at $v$ such that $v'\in \interior(S_{in}(k_i))\setminus \ell$, and both $\beta_m(\sigma^z, v')$ and  $\beta_m(\kk_i, v')$ are well-defined.
    Applying \reflem{basePointChange} and  \reflem{moveOneStrand} consecutively as above, we can also obtain the desired result. 

    Now, assume that $v\in S_{out}(\cO_z)$. Since $S_{out}(\cO_z)$ is path-connected, by \reflem{basePointChange}, 
    $\beta_m(\sigma^z, v)$ is conjugate to $\beta_m(\sigma^z, w_2^i)$ in $P_2(M)$. 
    As above, by \reflem{moveOneStrand},  $\beta_m(\sigma^z, w_2^i)$ is conjugate to $\beta_m(\kk_i, w_2^i)$ in $P_2(M)$.
    Thus, we can obtain the desired result.
\end{proof}
\end{case}

\begin{case}
$\sigma^z$ is an embedded path.\\
For each $h<D_i$, 
we denote by $\ell_{h}$ the leaf of $\cF$ that is at a distance of $h$ from $\partial S_{a_i}$.
Since $\theta_i$ is smooth, for any sufficiently small $h$, the orbit $\theta_i(\ell_h)$ is transverse to each leaf of the foliation $\cA_i$.
Hence, by taking a smaller $d$ if necessary, we may assume that for any $h\leq d$, the orbit $\theta_i(\ell_h)$ is transverse to each leaf of the foliation $\cA_i$.

We denote by $b_i$ the component of $\partial N_{\xi_i}(k_i)$ contained in $S_{in}(k_i)$.
Note that $b_i$ is also a smooth curve that is in minimal position with respect to $N^i$.
Also, $B_i$ denotes the annulus bounded by $b_i$ and $k_i$. 
Since $\sigma_t$ is $\xi_i$-supported, the support of $\sigma_1$ contained in $B_i$. Hence,  if $\sigma^z$ is an embedded path, then  $\sigma^z([0,1])$ lies in $B_i$ and it is transverse to each leaf of $\cA_i$ by the choice of $d$.

Before the estimation of the general case, we first prove the  prototypical cases.
\begin{claim}\label{Clm:totalWinding}
Let $c$ be a smooth representative of a two-sided simple closed curve in $N^i$ that is in minimal position with respect to $N^i$.
Let $\nn(c)$ be the maximum value of \[
\{\nn(c,Q):\text{$Q$ is either a junction or strip of $N^i$}\}.\]
Recall the definition of $\nn(c,Q)$ in \refconst{minimalCurve}.
If $\cc:[0,1]\to M$ is an embedded smooth path such that $\cc((0,1))\subset c \cap\hat{M}$ and $\beta_m(\cc, w)$ is well-defined for some $m\in\{1,2\}$ and $w\in \hat{M}\setminus c$, then 
$\beta_m(\cc, w)=A_{2,3}^n$ for some $n\in\ZZ$ with $|n|\leq\nn(c)+1$.  
\end{claim}
\begin{proof}
Assume that $\beta_1(\cc, w)$ is well-defined for some $w\in \hat{M}\setminus c$.
By \refrmk{diskBraid},  $\beta_1(\cc, w)=A_{2,3}^n$ for some $n\in \ZZ$.
Here, we think of $\hat{M}$ as $(-1/2,1/2)\times I$ in the universal cover.
Note that $|n|$ is the number of turns of the strand $c(\cc,\bar{z}_1)$ around the strand $c(w,\bar{z}_2)$ in $\hat{M}\times [0,1]$.

Write $w=(w_1,w_2)$. If $0\leq w_2$, then we set $r(w)$ as the geodesic ray $\{(w_1,w):w_2\leq w \leq 1/2\}$. We call the point $(w_1,1/2)$ the \emph{end} of $r(w)$. Likewise, if $w_2<0$, then we set $r(w)$ as the geodesic ray $\{(w_1,w):-1/2\leq w\leq w_2 \}$ and call $(w_1,-1/2)$ the \emph{end} of $r(w)$.

Observe that if the end of $r(w)$ is contained in a strip, then $r(w)$ intersects transversely each branch of $c$. Otherwise, by the monotone condition of $c$ in junctions (\refconst{minimalCurve}), 
each branch either intersects $r(w)$ along a unique arc or is transverse to $r(w)$.
Meanwhile, when closing up $(\cc,w)$ to $\beta_1(\cc,\bar{z}_1)$, one more turn can be introduced. 
It follows from the observation that $|n|\leq n_0+1$ where $n_0$ is the number of branches in a junction or a strip that contains the end of $r(w)$.

We can do the similar estimation for the case of $\beta_2(\cc, w)$. Thus, we are done. 
\end{proof}

Now, we consider the case of the paths lying in $B_i$.

\begin{claim}\label{Clm:shortPath}
    Let $\cc:[0,1]\to M$ be an embedded smooth path such that 
    $\cc((0,1))\subset \hat{M}$ and it is contained in an orbit of $\sigma_t$.
    Assume that $\cc([0,1])$ is contained the interior of $B_i$ and it is transverse to each leaf of $\cA_i$.
    If $\beta_m(\cc, w)$ is well-defined for some $m\in\{1,2\}$ and $w\in \hat{M}\setminus \cc([0,1])$, then $\beta_m(\cc, w)=A_{2,3}^n$ for some $n\in\ZZ$ with $|n|\leq\nn(k_i)+3$.
\end{claim}
\begin{proof}
By \refrmk{diskBraid},  $\beta_m(\cc, w)=A_{2,3}^n$ for some $n\in \ZZ$.
We write $\alpha_t$ for the leaf of $\cA_i$ containing $c(t)$. 
     Since $\cc(0)\neq \cc(1)$ and $\cc$ is transverse to each leaf of $\cA_i$,  we have that $\alpha_t\neq \alpha_s$ for any $s\neq t \in [0,1]$. 
    Hence, we can find  a smooth path $\dd_0:[0,1]\to M$ such that $\dd_0$ is a smooth curve perpendicular to each leaf of $\cA_i$ and there is an isotopy $F:[0,1]\times [0,1] \to M$ satisfying the following:
    \begin{itemize}
        \item $F(0,t)=\cc(t)$ and $F(1,s)=\dd_0(t)$;
        \item $F(s,t)\subset \alpha_t$ for all $(s,t)\in [0,1]\times [0,1]$;
        \item $F([0,1]\times [0,1])\subset M\setminus \{w\}$.
    \end{itemize}  
Note that $\dd_0([0,1])$ is a subarc of $\frak{d}$ that is the component of $\partial N_{\xi'}(k_i)$, contained in $S_{in}(k_i)$,  for some $\xi'\leq \xi_i$.
Moreover, $\frak{d}$ is a simple closed curve in minimal position with respect to $N_i$, lying on $B_i$.

Then, we construct $\dd(t)$ as the following:
\begin{itemize}
    \item $t\in[0,1/3] \mapsto s_{\cc(0),\dd_0(0)}(t)$;
    \item $t\in [1/3,2/3]\mapsto \dd_0(3t-1)$;
    \item $t\in[2/3,1] \mapsto s_{\dd_0(1),\cc(1)}(t)$.
\end{itemize}
Note that $s_{\cc(0),\dd(0)}$ and $s_{\dd(1),\cc(1)}$ are geodesic segments contained in $\alpha_0$ and $\alpha_1$, respectively.
Since $\beta_m(\cc,w)=\beta_m(\dd,w)$, $\nn(\frak{d})=\nn(k_i)$, and the ray $r(w)$, introduced in \refclm{totalWinding}, can intersect geodesic segments $s_{\cc(0),\dd(0)}$ and $s_{\dd(1),\cc(1)}$, 
we can see that $|n|\leq \nn(k_i)+3$.
\end{proof}

Now, we are ready to estimate the general case.
Choose $z\in \hat{M}$.
Assume that $\sigma^{z}$ is an embedded path with $\sigma^z(1)\in \hat{M}$ and $v$ is a point in $\hat{M}$. 
Hence, $\sigma^z([0,1])$ is a path, lying in $B_i$ and transverse to each leaf of $\cA_i$ by the choice of $d$ and  construction of $\sigma_t$ (\refconst{supportedTwist}).

We write $\fN_i=| k_i\cap \ell|$.
If $\fN_i=0$ and so $\sigma^z((0,1))\subset \hat{M}$, then by \refclm{shortPath}, $\sigma^z((0,1))\subset \hat{M}$, $\beta_m(\sigma^{z}, v)=A_{2,3}^n$  for some $n$ with $|n|\leq \nn(k_i)+3$ for almost every $(v,m)\in \hat{M}\times\{1,2\}$. 
Therefore, since $A_{2,3}=B_{2,3}^2$ and $B_{2,3}\in R$, we have \[
q_R(\beta_m(\sigma^{z}, v))\leq 2(\nn(k_i)+3)
\]
for almost every $(v,m)\in \hat{M}\times\{1,2\}$.  

Otherwise,  since $\fN_i\neq 0$, we can take a finite sequence of numbers, $0=s_0<s_1<\cdots<s_{\fN_i}<s_{\fN_i+1}=1$ such that 
\[
    k_i \cap \ell=\{\sigma^z(s_1),\sigma^z(s_2),\dots, \sigma^z(s_{\fN_i})\}.
 \]
Then, for almost every $(v,m)\in \hat{M}\times\{1,2\}$, $\beta_m(\sigma^z,v)$ are well-defined and 
    \[
    \beta_m(\sigma^z,v)= \beta_m(\sigma^z_0,v)\cdot \eta_m(\sigma^z(s_0),v)^{\epsilon_0}\cdot \beta_m(\sigma^z_1,v)\cdot\eta_m(\sigma^z(s_1),v)^{\epsilon_1}\cdot \ldots \cdot  \eta_m(\sigma^z(s_{\fN_i}),v)^{\epsilon_{\fN_i}}\cdot \beta_m(\sigma^z_{\fN_i},v)
    \]
    for some $\epsilon_i\in\{\pm1\}$
    where  each $\sigma^z_j:[0,1]\to M$ is a reparametrization of $\sigma^z|_{[s_j,s_{j+1}]}$, preserving the orientation.
    Therefore, by \refrmk{eachFactor}, \refclm{shortPath}, \reflem{auxBraid} and \refrmk{etaTerm}, 
    \[
q_R(\beta_m(\sigma^z_{j},v))\leq 2(\nn(k_i)+3) \text{ and }q_R(\eta_m(\sigma^z(s_j),v)^{\epsilon_j})\leq 
    3.\]
    Thus, we have
    \[
    q_R(\beta_m(\sigma^z, v))\leq 2(\nn(k_i)+3)(\fN_i+1)+3\fN_i=
    2\nn(k_i)\fN_i+2\nn(k_i)+9\fN_i+6
    \]
for almost every $(v,m)\in \hat{M}\times\{1,2\}$.
\end{case}

From the above case study, 
we have that 
\begin{align*}
   q_R(\gamma(\sigma_1;z))
   &\leq q_R(\beta_2(\sigma^{z_2},z_1))+q_R(\beta_1(\sigma^{z_1},\sigma_1(z_2)))\\
   & \leq 2(2\nn(k_i)\fN_i+2\nn(k_i)+9\fN_i+6+\max_{m,n\in \{1,2\}}\{q_R(\beta_m(\kk_i, w_n^i))\})
\end{align*}
for almost every $z\in \Omega^4$.
Set $\sigma_1=\tau_i$ and \[K_R=2(2\nn(k_i)\fN_i+2\nn(k_i)+9\fN_i+6+\max_{m,n\in \{1,2\}}\{q_R(\beta_m(\kk_i, w_n^i))\}).\]
Note that by the choice of $\{k_j\}_{j\in \NN}$, $\nn(k_j)=\nn(k_{j+1})$ and $\fN_j=\fN_{j+1}$ for all $j\in \NN$. 
Moreover, by \reflem{moveOneStrand} and \reflem{basePointChange},
\[\max_{m,n\in \{1,2\}}\{q_R(\beta_m(\kk_j, w_n^j))\}=\max_{m,n\in \{1,2\}}\{q_R(\beta_m(\kk_{j+1}, w_n^{j+1}))\}\] for all $j\in \NN$. 
Thus, we can conclude that  the sequence $\{\tau_j\}_{j\in \NN}$ of $\xi_j$-supported Dehn twists along $k_j$ satisfies the following inequality:
for each $j\in \NN$, \[
q_R(\gamma(\tau_j; z))\leq K_R
\]
for almost every $z\in \Omega^4$.
In the same way, by replacing $\{\sigma_t\}_{t\in [0,1]}$ with $\{\sigma_{-t}\}_{t\in [0,1]}$, we can also see that for each $j\in \NN$, \[
q_R(\gamma(\tau_j^{-1}; z))\leq K_R'
\]
for almost every $z\in \Omega^4$.
Thus, this implies the desired result.  
\end{proof}

\subsection{Ishida type argument for the injectivity of $\cG$}

By \refthm{wellDefineness}, it is shown that 
$\cG$ is a well-defined homomorphism from $Q(B_2(M))$ to $Q(\Diff_\omega(M,\partial M)_0)$ as $\RR$-vector spaces.
In this section, we show the injectivity of $\cG$, following the strategy outlined in \cite{Ishida14} and \cite{Brandenbursky15}. However, our proof is not identical.

\begin{thm}\label{Thm:injectivity}
    $\cG$ is injective.
\end{thm}
\begin{proof}
    Let $a_1,a_2,\epsilon$ be positive numbers such that $a_1+a_2\leq 1$ and $0<\epsilon<a_i^*/4$ for all $i=1,2$. 
    Assume that $\bar{z}_i\in \interior(V_i(a_1,a_2,\epsilon)) \cap \gamma_0$ for all $i=1,2$. 
    Set $R=\{B_{2,3},\rho_2,\rho_3\}$ as a finite c-generating set of $B_2(M)$.
    
    Let $\varphi$ be a non-trivial element in $Q(B_2(M))$. 
    Then, there is a braid $\beta$ in $B_2(M)$ such that $\varphi(\beta)\neq 0$.
    Since, by \refcor{equivalence},  $B_2(M)=\Mod(M,\{\bar{z}_1,\bar{z}_2\})$, there is a corresponding mapping class $\fB$ in  $\Mod(M,\{\bar{z}_1,\bar{z}_2\})$.
    By \refprop{twistGrp}, there is a non-trivial power $B^k\in \cT(M,\{\bar{z}_1,\bar{z}_2\})$.
    By the definition, $\cT(M,\{\bar{z}_1,\bar{z}_2\})$ is a subgroup of $\PMod(M,\{\bar{z}_1,\bar{z}_2\})$ (e.g. see \cite[A.~Appendix]{KatayamaKuno24}) and so $\beta^k\in P_2(M)$. 
     Since $\varphi(\beta^k)\neq 0$, without loss of the generality,
     we may assume that $\beta$ is a pure braid and $\fB\in \cT(M,\{\bar{z}_1,\bar{z}_2\})$.

     Now, we construct a diffeomorphism $g$ in $\Diff_\omega(M,\partial M)_0$ such that $g$ is a representative of $\fB$ and $\cG(\varphi)(g)\neq 0$. This implies the injectivity of $\cG$.

    Since $\fB\in \cT(M,\{\bar{z}_1,\bar{z}_2\})$, there is a finite collection $\{\gamma_1,\cdots, \gamma_n\}$ of two-sided simple closed curves in $M\setminus\{\bar{z}_1,\bar{z}_2\}$ such that each $\gamma_i$ is either peripheral or generic and $\fB=T_{\gamma_n} \circ \cdots \circ T_{\gamma_1}$, where $T_{\gamma_i}$ is the Dehn twist along $\gamma_i$.

    By \reflem{canonicalDehnTwist}, 
    for each $i\in\{1,2,\dots, n\}$, we can take a number $K_i>0$ and  a sequence $\{(\epsilon_{m,i}, k_{m,i}, \xi_{m,i}, \tau_{m,i})\}_{m\in \NN}$ satisfying the following:
        \begin{itemize}
            \item $\{\epsilon_{m,i}\}_{m\in\NN}$ is a strictly decreasing sequence of positive numbers such that $\epsilon_{1,i}<\epsilon$ and $\epsilon_{m,i} \to 0$ as $m\to \infty$;
            \item $\{k_{m,i} \}_{m\in\NN}$ is a sequence of representatives of $\gamma_i$ such that for each $m\in \NN$, $k_{m,i}$ is in minimal position with respect to $N(a_1,a_2,\epsilon_{m,i} )$;
            \item $\{\xi_{m,i} \}_{m\in\NN}$ is a sequence of positive numbers such that for each $m\in \NN$, $N_{\xi_{m,i} }(k_{m,i})$ is a tubular neighborhood of $k_{m,i} $ in $N(a_1,a_2,\epsilon_{m,i} )$;
            \item for each $m\in\NN$, $\tau_{m,i} $ is a $\xi_{m,i} $-supported Dehn twist along $k_{m,i}$;
            \item the following inequality holds:
            for any $m\in \NN$,
            \[q_S(\gamma(\tau_{m,i}^{\pm 1} ,z))\leq K_i\]
            for almost every $z\in \Omega^4$,
            where  $q_R$ is the norm c-generated by $R$.
        \end{itemize}

    For each $i\in \{1,2,\cdots, n\}$, there is an $e_i\in \{ \pm 1\}$ such that $\tau_{m,i}^{e_i}$ are  representatives of $T_{\gamma_i}$ in $\Diff_\omega(M,\partial M)_0$. 
    For any $m\in \NN$, we set $g_m=\tau_{m,n}^{e_n}\circ \tau_{m,n-1}^{e_{n-1}}\circ\cdots \circ \tau_{m,1}^{e_1}$.
    Each $g_m$ is a representative of $\fB$ in $\Diff_\omega(M,\partial M)_0$.
    We write \[V_j^m=\bigcap_{i=1}^n V_j(a_1,a_2,\epsilon_{m,i})\]
    for $j\in \{1,2\}.$
    Note that  for each $j\in\{1,2\}$, $\{V_j^m\}_{m\in\NN}$ is a nested increasing sequence and the area of $V_j^m$ converges to $a_j^*$ as $m\to \infty$.
    Also, $g_{m'}$ are the identity on $V_j^m$ for all $m'\geq m$.
    
    Set $\varphi_{ij}=\varphi(\gamma(g_m;z))$ for $z_1\in V_i^m$ and $z_2\in V_j^m$.
    Note that $\varphi_{ij}$ do not depend on $m$.
    Now, we consider the following polynomial in $\RR[x,y]$,
\[
P(x,y)=\varphi_{11}x^2 + \varphi_{12}xy +\varphi_{21}yx+ \varphi_{22}y^2.
\]
Since $\varphi_{12}=\varphi(\beta)\neq 0$ and $\varphi_{12}=\varphi_{21}$ by the invariance of $\varphi$ under conjugation (\refprop{conjInv}),
\[
P(x,y)=\varphi_{11}x^2 + 2\varphi_{12}xy + \varphi_{22}y^2
\]
and it is not  identically $0$.
Therefore, there are positive numbers $c_1$ and $c_2$ such that $c_1+c_2< 1$ and $P(c_1,c_2)\neq 0$.
Then, we replace $a_i$ by $c_i$ and rechoose the numbers $\epsilon,K_i>0$ and  a sequence $\{(\epsilon_{m,i}, k_{m,i}, \xi_{m,i}, \tau_{m,i})\}_{m\in \NN}$ as above.
Observe that  $\varphi_{ij}$ are invariant under the replacement and $P(a_1,a_2)\neq0$.

Set \[\fK=K_n+K_{n-1}+\cdots+K_1\text{ and } \fM=\max\{|\varphi(B_{2,3}^{\pm 1})|, |\varphi(\rho_2^{\pm 1})|,|\varphi(\rho_3^{\pm 1})|\}.\]
Since there is an $f>0$ such that 
$|P(b_1,b_2)|\geq f$ for any $b_1$, $b_2$ with $a_i\leq b_i$ and $b_1/b_2=a_1/a_2$, we can choose $b_1,b_2>0$ such that $b_1+b_2<1$, $a_i\leq b_i$, $b_1/b_2=a_1/a_2$  and
\[
\fK(\fM+D(\varphi))(1-(b_1+b_2)^2)<|P(b_1,b_2)|
\]
where $D(\varphi)$ is  the defect of $\varphi$.

Since  $b_1+b_2<1$ and $b_1/b_2=a_1/a_2$, the area of each $V_i^m$ is greater than $b_i$ for any sufficiently large $m$. Fix such a $m$.
Let $\{U_1,U_2\}$ be a pair of disjoint open subsets of $\hat{M}$ such that $\bar{z}_i\in U_i \subset \interior(V_i^m)$ and each $U_i$ is a topological disk with area $b_i$.
See \reffig{decomposition}.
Note that the support of $g_m$  does not intersect $U=U_1\cup U_2$. 

Now, we claim that $\cG(\varphi)(g_m)\neq 0$.
Consider 
\begin{align*}
\cG(\varphi)(g_m)&=\lim_{p\to \infty} \frac{1}{p}\int_{X_2(M)}\varphi(\gamma(g_m^p;z))dz \\
&=\lim_{p\to \infty} \frac{1}{p}
\left(
\int_{X_2(U)}\varphi(\gamma(g_m^p;z))dz+
\int_{X_2(M)\setminus X_2(U)}\varphi(\gamma(g_m^p;z))dz
\right).\\
\end{align*}

First, we consider the first term \[
\fF=\lim_{p\to \infty} \frac{1}{p}
\int_{X_2(U)}\varphi(\gamma(g_m^p;z))dz.
\] Since $g_m$ is the identity on $U$ and $\gamma(g_m^p;z)=\gamma(g_m^{p-1};z)\cdot\gamma(g_m;z)$ for all $z\in X_2(U)$, 
we have that
\[\fF=\int_{X_2(U)}\varphi(\gamma(g_m;z))dz.\]
Then, we can write   
\begin{equation}\label{Eqn:core}
\fF= \varphi_{11}b_1^2 + \varphi_{12}b_1b_2 +\varphi_{21} b_2b_1+ \varphi_{22}b_2^2=P(b_1,b_2).
\end{equation}

Then, we consider the second term,
\[
\fR=\lim_{p\to \infty} \frac{1}{p} \int_{X_2(M)\setminus X_2(U)}\varphi(\gamma(g_m^p;z))dz.
\]
Since $\varphi$ is homogeneous and for each $i\in\{1,2,\dots, n\}$ and for any $m\in \NN$,
\[q_S(\gamma(\tau_{m,i}^{e_i} ,z))\leq K_i\]
for almost every $z\in \Omega^4$,
we have that 
for each $m\in \NN$,
\begin{align*}
q_R(\gamma(g_m;z))&=q_R(\gamma(\tau_{m,1}^{e_1};z)\cdot \gamma(\tau_{m,2}^{e_2};\tau_{m,1}(z)) \cdot\dots \cdot \gamma(\tau_{m,n}^{e_n};\tau_{m,n-1}^{e_{n-1}}\cdots\tau_{m,2}^{e_2}\tau_{m,1}^{e_1}(z))))\\
&\leq q_R(\gamma(\tau_{m,1}^{e_1};z))+q_R( \gamma(\tau_{m,2}^{e_2};\tau_{m,1}^{e_1}(z)))+ \dots+ q_R(\gamma(\tau_{m,n}^{e_n};\tau_{m,n-1}^{e_{n-1}}\cdots\tau_{m,2}^{e_2}\tau_{m,1}^{e_1}(z)))\\
&\leq K_1+K_2+\cdots+K_n\\
&=\fK
\end{align*}
for almost every $z\in \Omega^4$.
Therefore, 
\begin{align*}
q_R(\gamma(g_m^p;z))&=q_R(\gamma(g_m;z)\cdot \gamma(g_m;g_m(z))\cdot \ldots \cdot\gamma(g_m;g_m^{p-1}(z))))\\
&\leq q_R(\gamma(g_m;z))+q_R( \gamma(g_m;g_m(z)))+ \cdots +q_R(\gamma(g_m;g_m^{p-1}(z)))\\
&\leq p\fK
\end{align*}
for almost every $z\in \Omega^4$.
Then, we have that
\[
|\varphi(\gamma(g_m^p;z))|\leq p\fK(\fM+D(\varphi))
\]
for almost every $z\in \Omega^4$ where $D(\varphi)$ is the defect of $\varphi$.
Therefore, 
\begin{equation}\label{Eqn:remainder}
|\fR|\leq \fK(\fM+D(\varphi)) \cdot vol(X_2(M)\setminus X_2(U))=\fK(\fM+D(\varphi)) (1-(b_1+b_2)^2)
\end{equation}
where $vol(X_2(M)\setminus X_2(U))$ is the volume of $X_2(M)\setminus X_2(U)$ in $X_2(M)$. 

Since \[
\fK(\fM+D(\varphi))(1-(b_1+b_2)^2)<|P(b_1,b_2)|,
\] by \refeqn{core} and \refeqn{remainder}, 
we can see that 
\[\cG(\varphi)(g_m)=\fF+\fR\neq0.\]
This shows the injectivity of $\cG$.
\end{proof}

\begin{thm}\label{Thm:inftyDim}
    The group $\Diff_{\omega}(M,\partial M)_0$ admits countably many homogeneous quasimorphisms which are linearly independent.
\end{thm}
\begin{proof}
    It is a combination of \reflem{inftyDimBraid} and \refthm{injectivity}.
\end{proof}

\section{Boundedness of the word length of the cocycle $\gamma$}\label{Sec:wordLen}

In the proof of  \refthm{wellDefineness} and \refthm{injectivity}, we used \reflem{finiteness} without providing a proof.
In this section, we prove \reflem{finiteness}. 
To do this, we first introduce some compactification  $\closure{X}_2(M)$ for $X_2(M)$, which is a sort of blowing up the diagonal $M^{\times 2}$. 
This is a modification of the blowing-up set, introduced in the proof of \cite[Proposition~2]{GambaudoPecou99}.
Our  blowing-up set  is homotopy equivalent to the configuration space unlike the  blowing-up set in \cite[Proposition~2]{GambaudoPecou99}. 

\subsection{The injectivity radius of the M\"obius band}
To construct a well-defined compactification, we need the concept of the injectivity radius of a Riemannian manifold.
Unlike closed Riemannian manifolds, the injectivity radius of a Riemannian manifold with non-empty boundary is not well defined near the boundary.
Hence, we need to modify the definition of the injectivity radius. 
We follow  a version of the injectivity radius, used in \cite{BuragoIvanovLassasLu24}.
See \cite[Section 2.1]{BuragoIvanovLassasLu24}.
Instead of introducing a general definition of the injectivity radius for a non-orientable Riemannian manifold with boundary, for the simplicity, we only introduce the injectivity radius of our M\"obius band $M$. Also, we define a version of an exponential map at each point in $M$.

Recall that we use the Riemannian metric, inherited from the Euclidean metric on the universal cover. 
We consider $\widetilde{M}$ as a subset of $\RR^2$ and also $\tau$ is extended on $\RR^2$ in the obvious way.
Now, we define the \emph{injectivity radius} $\inj(M)$ of $M$ as the largest number $r>0$ satisfying the following condition:
the open $r$-ball $B_r(x)$ at $x$ in $\RR^2$ does not intersect $\tau^n(B_r(x))$  for any point $x\in \widetilde{M}$ and for all $n\in \ZZ\setminus \{0\}$. 
Observe that $\inj(M)=1/2$.

Say $M_{ext}=\RR^2/\langle \tau \rangle$. Also, $M_{ext}$ is equipped with the Riemannian metric induced from the Euclidean metric in $\RR^2$.
Then, we can see that for each $p\in M_{ext}$, the \emph{exponential map} $\exp_p^{ext}$ at $p$ in $M_{ext}$ is well-defined near $p$ as follows. For any $r< \inj(M)$, there is a diffeomorphism $\exp_p^{ext}$ from the open $r$-ball $B_r(0)$ in $T_pM_{ext}$ to the open $r$-neighborhood $N_r(p)$ of $p$ in $M_{ext}$ defined as follows:
for any $v\in B_r(0)$, there is a unique geodesic $\gamma_v:[0,1] \to M_{ext}$ satisfying $\gamma_v(0) = p$ with initial tangent vector $\gamma_v'(0) = v$. We define $\exp_p^{ext}:B_r(0)\to N_r(p)$ as $\exp_p^{ext}(v)=\gamma_v(1)$.

Note that $M$ is a submanifold of $M_{ext}$, the boundary of which is a geodesic. 
Fix $r$ with $0<r\leq1/2$.
For each $p\in M$, the open $r$-neighborhood of $p$ in $M$ is $N_r(p)\cap M$. 
If $p$ is not contained in the open $r$-neighborhood of the boundary $\partial M$, then $N_r(p)\cap M=N_r(p)$. Otherwise, $N_r(p)\cap M\neq N_r(p)$. 
In this case, there is a unique closed half-plane $H_p$ in $T_pM=T_pM_{ext}$ such that $(\exp_p^{ext})^{-1}(N_r(p)\cap M)=B_r(0)\cap H_p$. 
Therefore, for each $p\in M$ and for any $v\in B_r(0)\cap H_p$, $\exp_p^{ext}(v)$ is a well-defined point in $M$.

\begin{rmk}
    Note that the half-plane $H_p$ does not depend on $r$. 
\end{rmk}

By the remark, for any $p$ in the open $1/2$-neighborhood of $\partial M$, we can find a well-defined  half-plane $H_p$ such that $(\exp_p^{ext})^{-1}(N_r(p)\cap M)=B_r(0)\cap H_p$ for all $0<r\leq 1/2$ . 
We call $H_p$ the \emph{defining half-plane} at $p$.
If $p\in \partial M$, then the boundary of the defining half-plane is a line passing through $0$.

Now, we define the \emph{exponential map} $\exp_p$ at $p$ in $M$ as follows: if $p$ is not in the open $1/2$-neighborhood of $\partial M$, then we define $\exp_p:B_{1/2}(0)\to N_{1/2}(p)$ as $\exp_p=\exp_p^{ext}$. Otherwise, we define $\exp_p:B_{1/2}(0)\cap H_p\to N_{1/2}(p)\cap M$ by restricting the domain and range of the exponential map $\exp_p^{ext}:B_{1/2}(0)\to N_{1/2}(p)
$ onto $B_{1/2}(0)\cap H_p$ and $N_{1/2}(p)\cap M$, respectively. 

\subsection{Blowing up $\Delta_2(M)$}

Inspired by the blowing-up set $\cK$ of the generalized diagonal in $\closure{\DD}\times \cdots \times \closure{\DD}$, introduced in the proof of \cite[Proposition 2]{GambaudoPecou99}, we compactify  $X_2(M)$ by blowing up the diagonal $\Delta=\Delta_2(M)$ in $M\times M$ so that $\Diff^1(M)$ acts continuously on the compactification.

For $\epsilon\geq 0$, we define  
$\Delta(\epsilon)$ and $\delta(\epsilon)$ as 
$$\Delta(\epsilon)=\{(p_1,p_2)\in M\times M: d(p_1,p_2)\leq \epsilon\}$$
and 
$$\delta(\epsilon)=\{(p_1,p_2)\in M\times M: d(p_1,p_2)=\epsilon\}$$
where $d$ is the Euclidean metric.
Note that $\Delta(\epsilon)$ and $\delta(\epsilon)$ are closed sets and $\Delta(0)=\delta(0)=\Delta$.
We also define $\Delta^+(\epsilon)=\Delta(\epsilon)\setminus \Delta$.

Observe that if there is a sequence $\{(p_n,q_n)\}_{n\in \NN}$ in $X_2(M)$ such that $\{p_n\}_{n\in \NN}
$ and $\{q_n\}_{n\in \NN}$ are Cauchy sequences, then $p_n\to p$ and $q_n\to q$ for some $p$ and $q\in M$ as $M$ is compact. If $p\neq q$, then $\{(p_n,q_n)\}_{n\in \NN}$ converges to a point in $X_n$. Otherwise, $p=q$ and $\{(p_n,q_n)\}_{n\in \NN}$ approaches the diagonal $\Delta$ as $n\to \infty$.
Therefore, once we find a good compactification of $\Delta^+(\epsilon)$ for some $0<\epsilon<\inj(M)$, it provides a desired compactification of $X_2(M)$.

Choose $\epsilon$ with $0<\epsilon<1/2$. Note that $\inj(M)=1/2$.
We define the \emph{blow-up} $\fB\Delta(\epsilon)$ of $\Delta(\epsilon)$ as the collection of all triples $(p,q,R)$ such that 
$(p,q)\in \Delta(\epsilon)$ and $R$ is a ray in $T_pM$, starting at $0$ and passing through $\exp_p^{-1}(q)$.
Note that if $(p,q,R)\in \fB\Delta(\epsilon) \setminus \Delta^+(\epsilon)$, then $p=q$. In this case, $R$ can be any ray in $T_pM$ starting at $0$.

To assign a reasonable topology of the blow-up $\fB\Delta(\epsilon)$, we consider an embedding ${\Bl_\epsilon}$ of  $\fB\Delta(\epsilon)$ into the tangent bundle $TM$ defined as follows:
let $(p,q,R)$ be a point in $\fB\Delta(\epsilon)$ and $v_R$ the unit vector in $R$, which is unique.
Then, we set  $\Bl_\epsilon(p,q,R)=e^{d(p,q)}v_R\in T_pM$ where $d(p,q)$ is the distance between $p$ and $q$ in $M$.
Via the embedding $\Bl_\epsilon$, we think of the blow-up $\fB\Delta(\epsilon)$ as a subspace of $TM$. 
Therefore, by taking the subspace topology, we can introduce a natural topology for  $\fB\Delta(\epsilon)$. Observe the following proposition:
\begin{prop}\label{Prop:blowUpCompact}
$\fB\Delta(\epsilon)$ is compact. 
\end{prop}

On the other hand, $\Delta^+(\epsilon)$ can be naturally embedded in $\fB\Delta(\epsilon)$ in the following way.
For each $(p,q)\in \Delta^+(\epsilon)$, there is a unique ray $R_{pq}$ in $T_pM$ such that $R_{pq}$ starts at $0$ and  passes through $\exp_p^{-1}(q)$.
Therefore, $\Delta^+(\epsilon)$ is naturally embedded in $\fB\Delta(\epsilon)$ by $(p,q) \mapsto (p,q, R_{pq})$. Say that the embedding is $\iota_\epsilon:\Delta^+(\epsilon)\to \fB\Delta(\epsilon)$.
If there is no confusion, then we do not strictly distinguish the image of $\iota_\epsilon$ with $\Delta^+(\epsilon)$.

Recall that if $(p,q,R)\in \fB\Delta(\epsilon) \setminus \Delta^+(\epsilon)$, then $p=q$ and $R$ can be any ray in $T_pM$ starting at $0$. Hence, the following proposition follows.
\begin{prop}
$\Bl_\epsilon(\fB\Delta(\epsilon)\setminus\Delta^+(\epsilon))$ is the unit tangent bundle $T^1M$ of $M$.
\end{prop}
Since every element of $\fB\Delta(\epsilon)$ can be approximated by elements of $\Delta^+(\epsilon)$, we also have the following proposition.
\begin{prop}\label{Prop:denseInBlowUp}
    $\Delta^+(\epsilon)$ is a dense, open subset of $\fB\Delta(\epsilon)$.
\end{prop}

Finally, we remark the following:
\begin{prop}\label{Prop:wellDefineLimit}
    For any $\epsilon_1,\epsilon_2$ with $0<\epsilon_1<\epsilon_2<\inj(M)$,  we have that $\fB\Delta(\epsilon_1)\subset \fB\Delta(\epsilon_2)$. Moreover, \[\bigcap_{0<\delta<\inj(M)}\fB\Delta(\delta)=\fB\Delta(\epsilon)\setminus\Delta^+(\epsilon)\]
    for any $\epsilon$ with $0<\epsilon < \inj(M)$.
\end{prop}

\subsection{Compactification of $X_2(M)$}

Choose $\epsilon$ with $0<\epsilon<1/2$.
We define  the \emph{compactification} $\closure{X}_2(M)$ of $X_2(M)$ as the attaching space $\fB\Delta(\epsilon)\bigcup_{\iota_\epsilon}{X}_2(M)$ by the attaching map $\iota_\epsilon$. 
In other words, we attach $x\in \Delta^+(\epsilon)\subset \fB\Delta(\epsilon)$ to $\iota_\epsilon(x)\in \iota_\epsilon(\Delta^+(\epsilon))\subset\closure{X}_2(M)$.
\begin{rmk}
    We think of $\fB\Delta(\epsilon)$ and ${X}_2(M)$ as subspaces of $\closure{X}_2(M)$.
\end{rmk}
By \refprop{wellDefineLimit}, $\closure{X}_2(M)$ does not depend on $\epsilon$.
Moreover, the following proposition follows from \refprop{blowUpCompact} and \refprop{denseInBlowUp}.

\begin{prop}\label{Prop:compact}
    $\closure{X}_2(M)$ is compact and ${X}_2(M)$ is a dense open subset of $\closure{X}_2(M)$.
\end{prop}

Now, we claim that the blowing-up of the diagonal does not change the topology of $X_2(M)$.
\begin{lem}\label{Lem:homEqu}
$X_2(M)$ and $\closure{X}_2(M)$ are homotopy equivalent.
\end{lem}
\begin{proof}
    Observe that $X_2(M)\setminus \delta(\epsilon)$ has exactly two components. One of the components is $\Delta^+(\epsilon)\setminus \delta(\epsilon)$. Note that the closure of $\Delta^+(\epsilon)\setminus \delta(\epsilon)$ in $X_2(M)$ is $\Delta^+(\epsilon)$. We denote the closure of the other component by $\fC$.

    Now, we consider the embedding $\Bl_\epsilon:\fB\Delta(\epsilon)\to TM$.
    For a connected subset $I$ of $\RR_{\geq0}$,
    we denote by $T^I M$ the set of all vectors $v\in TM$ such that $|v|\in I$.
    In particular, if $I$ is $\{p\}$ for some $p\geq 0$, then we just write $T^pM$.
    
    Observe that $\Bl_\epsilon(\delta(\epsilon))$ is a subset of $T^{d}M$ where $d=e^\epsilon$. Also, $\Bl_\epsilon(\fB\Delta(\epsilon))$ is a subset of $T^{[1,d]}M$.
    Then, we construct $\fX$ by attaching $T^{(0,d]}M$ to $\fC$ along $\delta(\epsilon)$ with $\Bl_\epsilon$.
    We still think of $\Bl_\epsilon(\fB\Delta(\epsilon))$ as a subspace of $\fX$.
    The homotopy equivalency follows from the fact that $X_2(M)$ and $\closure{X}_2(M)$ are  deformation retracts of $\fX$.
\end{proof}

Recall that for each $h$ in $\Diff^1(M)$, $\bar{h}$ acts continuously on $M\times M$ and on $X_2(M)$. Now, we show that $\bar{h}$ can be extended to a homeomorphism on $\closure{X}_2(M)$. 
For each $(p,p,R)\in \fB\Delta(\epsilon)\setminus \Delta^+(M) \subset \closure{X}_2(M)$, we  define 
\[\bar{h}((p,p,R))=(h(p),h(p),dh_p(R)).\]
The continuity of the extension follows from the fact that for any $\delta$ with $0<\delta<\epsilon$, if a sequence $\{(x_n,y_n,L_n)\}_{n\in\NN}$ in $\fB\Delta(\delta)$ converges to $(x,x,L)$, then the sequence of the unit vectors $v_{L_n}$ of $L_n$ converges to  the unit vector $v_L$ in $L$ in the tangent bundle $TM$, and for each $(x,y,L)\in \fB\Delta(\delta)$, if $x\neq y$, then $L$ is uniquely determined.
\begin{prop}\label{Prop:extension} $\Diff^1(M)$ acts continuously on $\closure{X}_2(M)$. Namely, there is a continuous embedding from  $\Diff^1(M)$ to  $\Homeo(\closure{X}_2(M))$ defined by $h\mapsto \bar{h}$.
\end{prop}

\subsection{Boundedness of word lengths}

Recall the notions in \refsec{GGcocycle}.
The following lemma is a variation of \cite[Proposition~2]{GambaudoPecou99}. 
Note that $P_2(M)$ is finitely generated (e.g. see \cite{GoncalvesGuaschi17}).
\begin{restate}{Lemma}{Lem:finiteness}
    If $f \in \Diff_\omega(M,\partial M)_0$ and $S$ is a finite generating set of $\pi_1(X_2(M),\bar{z})$ where $\bar z \in X_2(\hat{M})$, then there is a constant $K(f, S)$ such that 
$$\ell_S(\gamma(f;z))\leq K(f, S)$$
for almost every $z$ in $\Omega^4$.
\end{restate}

\begin{proof}
    We consider the compactification $\closure{X}_2(M)$ of $X_2(M)$. By \refprop{compact}, $\closure{X}_2(M)$ is a compact and $X_2(M)$ is a dense open subset. 
    Moreover, by \reflem{homEqu}, $X_2(M)$ and $\closure{X}_2(M)$ are homotopy equivalent. Hence, we can think of $S$ as a finite generating set of $G=\pi_1(\closure{X}_2(M),\bar{z})$.

    We choose an isotopy $f_t$ from the identity to $f$ in $\Diff_\omega(M,\partial M)_0$. Then, by \refprop{extension}, there is a corresponding isotopy $\bar{f}_t$ from the identity to $\bar f$ in $\Homeo(\closure{X}_2(M))$.
    
    Now, we consider the continuous map $H:[0,1]\times \closure{X}_2(M) \to \closure{X}_2(M)$, given by 
    $H(t,x)=\bar{f}_t(x)$.
    Let $\widetilde{\fX}$ be the universal cover of $\closure{X}_2(M)$ and $q:\widetilde{\fX}\to \closure{X}_2(M)$ the covering map.
    Then, we take the lifting $\widetilde{H}$ of $H$ such that $\widetilde{H}:[0,1]\times \widetilde{\fX} \to \widetilde{\fX}$ is an isotopy from the identity to a lifting $\tilde{f}$ of $f$, that is, $\widetilde{H}(0,x)=x$, $\widetilde{H}(1,x)=\tilde{f}$, and
    $q(\widetilde{H}(t,x))=H(t,q(x))$.

    Recall that $\Omega^4$ is an open, dense subset of $X_2(M)$ and it is also contractible by the definition. By \refprop{compact}, $\Omega^4$ is also an open, dense and contractible subset of $\closure{X}_2(M)$.
    Fix a point $\tilde{z}\in \widetilde{\fX}$ such that 
    $q(\tilde{z})=\bar{z}$.
    We denote by $\fW$ the component of $q^{-1}(\Omega^4)$ containing $\tilde{z}$. 
    
    By the construction of $\gamma(f;\cdot)$, it is enough to show that 
    \[
        A = \{ g \in G \mid g(\fW) \cap \tilde{f}(\closure{\fW}) \neq \emptyset \}
    \]
    is finite.
    For contradiction, we assume that the $A$ is infinite.
    We choose $x_g \in g(\fW) \cap \tilde{f}(\closure{\fW})$ for every $g \in A$.
    By the compactness and metrizability of $\tilde{f}(\closure{\fW})$, there exists an accumulation point $\tilde{x} \in \tilde{f}(\closure{\fW})$ of $\{x_g \mid g \in A \}$.
    Set $x = q(\tilde{x}) \in \closure{X}_2(M)$.

    Since $q \colon \widetilde{\fX} \to \closure{X}_2(M)$ is a covering map, we take an open neighborhood $B \subset \closure{X}_2(M)$ of $x$ such that $q^{-1}(B)$ is the disjoint union of $\{ g(\widetilde{B}) \mid g \in G \}$, where $\widetilde{B} \subset \widetilde{\fX}$ is a homeomorphic lift of $B$ containing $\tilde{x}$.
    We set $S = \{ g \in G \mid x_g \in \widetilde{B} \}$, which is an infinite set.
    Note that $\{ g^{-1}(x_g) \mid g \in S\}$ is a closed subset of $\closure{\fW}$ since it has no accumulation point.

    We set $O = \closure{\fW} \setminus \{ g^{-1}(x_g) \mid g \in S\}$ and $O_g = g^{-1}(\widetilde{B})$ for every $g \in S$.
    Then $\{ O \} \cup \{ O_g \mid g \in S\}$ provides an open cover of $\closure{\fW}$ but does not admit a finite subcover, which is a contradiction.
\end{proof}

\section*{Acknowledgements}
We would like to thank Takashi Tsuboi, Sangjin Lee, Hongtaek Jung, Mitsuaki Kimura and Erika Kuno for helpful conversations and comments.
The first author was supported by the National Research  Foundation of Korea Grant funded by the Korean Government (RS-2022-NR072395).
The second author is partially supported by JSPS KAKENHI Grant Number JP23KJ1938 and JP23K12971.

\bibliographystyle{alpha}
\bibliography{biblio}

\end{document}